\documentclass[12pt]{amsart}
\usepackage{amsmath, amssymb,mathrsfs,color,enumerate,comment,hyperref,graphicx,float}
\newcommand{\no}[1]{#1}
\renewcommand{\no}[1]{}
\no{\usepackage{times}\usepackage[subscriptcorrection, slantedGreek, nofontinfo]{mtpro}
\renewcommand{\Delta}{\upDelta}}
\usepackage[foot]{amsaddr}

\title[Determination of Quasilinear Terms for Parabolic Equations]{Simultaneous Stable Determination of Quasilinear terms for Parabolic equations}
\date{\today}

\author{Jason Choy$^1$ \and Yavar Kian$^2$}

\newtheorem{Thm}{Theorem}[section]
\newtheorem{prop}{Proposition}[section]
\newtheorem{lem}{Lemma}[section]
\newtheorem{definition}{Definition}[section]

\theoremstyle{remark}
\newtheorem{remark}{Remark}[section]

\newcommand{\R}{{\mathbb R}}

\numberwithin{equation}{section}

\renewcommand{\leq}{\leqslant}
\renewcommand{\geq}{\geqslant}

\def\epsilon{\varepsilon}
\def\phi {\varphi}
\def\supp{\text{supp}}

\allowdisplaybreaks

\begin{document}

\begin{abstract}
In this work, we consider the inverse problem of simultaneously recovering two classes of quasilinear terms appearing in a parabolic equation from boundary measurements. It is motivated by several industrial and scientific applications, including the problems of heat conduction and population dynamics, and we study the issue of stability. More precisely, we derive simultaneous Lipschitz and H\"older stability estimates for two separate classes of quasilinear terms. The analysis combines different arguments including the linearization technique with a novel construction of singular solutions and properties of solutions of parabolic equations with nonsmooth boundary conditions. These stability results may be useful for deriving the convergence rate of numerical reconstruction schemes.\\

\medskip
\noindent
{\bf  Keywords:} Inverse problems, Stability estimates, Nonlinear PDE, Quasilinear parabolic equations, Heat conduction. \\

\medskip
\noindent
{\bf Mathematics subject classification 2020 :} 35R30, 35K59, 35K20.
\end{abstract}

\maketitle

\renewcommand{\thefootnote}{\fnsymbol{footnote}}
\footnotetext{\hspace*{-5mm} 
\begin{tabular}{@{}r@{}p{13cm}@{}} 
\\
$^1$& Department of Mathematics, The Chinese University of Hong Kong, Shatin, N. T. , Hong Kong; Email: zhchoy@math.cuhk.edu.hk\\
$^2$& Universit\'e de Rouen Normandie, CNRS UMR 6085, Math\'ematiques, 76801 Saint-Etienne du Rouvray, France; Email: yavar.kian@univ-rouen.fr\\
 \end{tabular}}
\section{Introduction}
In this article we consider the following parabolic initial  boundary value problem (IBVP) 
\begin{equation}
\label{IBVP}
    \begin{cases}
\rho(t,u)\partial_tu- \nabla \cdot (\gamma(t,u)A(x)  \nabla u)=0  & \text{in } \Omega\times (0,T):=Q ,
\\
u=\lambda+g &\text{on } \partial\Omega\times (0,T):=\Sigma,\\
u(x,0)=\lambda & x\in\Omega,
\end{cases}
\end{equation}
where  $0<T<\infty$, $\alpha\in(0,1)$, $\lambda \in \mathbb{R}$ is a constant, and $\Omega \subset \R^n$ with $n\geq 2$ is a bounded domain with a $\mathcal{C}^{2+\alpha}$ boundary. $A(x):=(a_{ij}(x))_{1 \leq i,j \leq n}$ is a matrix of real-valued functions with $a_{ij} \in \mathcal{C}^3(\overline{\Omega})$, and it is symmetric and elliptic, i.e., there exists $c>0$ such that
\begin{equation}
\label{ellipticity}
 a_{ij}(x)=a_{ji}(x) , \, \xi^T A(x) \xi \geq  c|\xi|^2,\quad 
 x \in \overline{\Omega}, \ \xi\in \R^n.
\end{equation}
In addition, the nonlinear diffusion and weight terms $\gamma,\rho \in \mathcal{C}^1([0,T];\mathcal{C}^3(\R))$ satisfy 
\begin{equation}
\label{positivity condition of gamma,rho}
    \inf_{t\in[0,T]}\gamma(t,s)\geq m(s) , \, \inf_{t\in[0,T]}\rho(t,s) \geq m(s),\quad s\in \R,
\end{equation}
for some strictly positive continuous function $m:\mathbb{R} \to (0,\infty)$. Moreover, by fixing $S$ to be an arbitrary open and not-empty subset of $\partial\Omega$, we define the partial parabolic Dirichlet-to-Neumann map (DN map in short)
\begin{equation}\label{eqn:defn of N}
\mathcal{N}_{\lambda,\gamma,\rho}: g  \mapsto \gamma(t,u_{\lambda,g}) (A(x)\nabla u_{\lambda,g} \cdot \nu(x))|_{S\times(0,T)} , \end{equation}
where $\nu(x)$ denotes the unit outward normal vector to the boundary $\partial\Omega$ and $u_{\lambda,g}$ is the solution of problem \eqref{IBVP}. Physically, the data 
encapsulated in the map $\mathcal{N}_{\lambda,\gamma,\rho}$ corresponds to measurements of the flux located at $S$ along the time interval $(0,T)$ associated with different boundary sources $\lambda + g$ applied on $\partial\Omega$ and a constant initial condition $u|_{t=0}=\lambda$. In this work, we study the inverse problem of simultaneously determining the quasilinear terms $\gamma$ and $\rho$ from a knowledge of $\mathcal{N}_{\lambda,\gamma,\rho}$, and we look for stability results.

The equation in \eqref{IBVP} is often used for modelling diverse physical phenomena, e.g., nonlinear heat conduction, chemical mixing, and population dynamics. The inverse problem is motivated by different physical and industrial applications in which the goal is to determine the underlying physical law of the system \eqref{IBVP} described by the quasilinear terms $\gamma$ and $\rho$. It is closely related to the problem of identifying thermal properties dependent on temperature and time, including thermal conductivity $\gamma$ and volumetric heat
capacity $\rho$; see \cite{Al,BBC,Schuster} for more details. This includes applications to microstructure adjustment in water cooling processes, where various macroscopic material properties described by the quasilinear terms $\gamma$ and $\rho$ are key targets in customer requirements \cite{Schuster,SKFS}. An in-depth study of the parameters $\gamma$ and $\rho$ is imperative to understand cooling processes on an industrial scale \cite{serizawa2015plate}.

The concerned inverse problem falls into the category of identification of nonlinear terms appearing in a PDE, which has received much interest in the mathematical community. For parabolic equations, this mostly includes results on the unique determination of semilinear terms and their applications in physical phenomena described by reaction diffusion equations \cite{EPS2,Is,Is1,Is2,KLL,KU}. Several works have also been devoted to the unique determination of quasilinear terms similar to the thermal conductivity $\gamma$. This includes earlier works \cite{CD1,CD2} in one space dimension and more recent contributions \cite{EPS2,FKU} in a higher space dimension with measurements given by the parabolic DN map $\mathcal{N}_{\lambda,\gamma,\rho}$, with $S=\partial\Omega$. The determination of quasilinear terms has been extensively investigated for elliptic equations, and without being exhaustive, we mention the pioneering works of \cite{Ca,PR} for heat conduction, but also the works \cite{CFKKU,CLLO,MU,Nu1,Nu2,Su,SuU} on the Calder\'on problem or the minimal surface equation. 

All aforementioned results are uniqueness results. For parabolic equations, we are only aware of the results \cite{CK,COY} devoted to the stable determination of a semilinear term in a parabolic equation. For elliptic equations, in \cite{EPS3} the authors proved uniqueness and stability results associated with the approach of \cite{Ca} and the stable determination of a more general class of quasilinear terms has been recently investigated in \cite{Cho,kian2023lipschitz,Ki24}. To the best of our knowledge, there is no result in the existing literature dealing with the stable determination of quasilinear terms appearing in a parabolic equation from boundary measurements. 

The remainder of this article is organized as follows. In Section \ref{sec:main} we describe the main results and provide relevant discussions. In Section \ref{Forward problem and DN map}, we discuss the well-posedness of the IBVP \eqref{IBVP} and the corresponding linearized problem. We also carefully define the partial parabolic DN map \eqref{eqn:defn of N} and introduce a second partial parabolic DN map \eqref{eqn: defn of linear dn} associated with the linearized problem. In Section \ref{section: construction of special solutions}, we construct novel singular solutions to the linearized problem, which is one of the key ingredients in the proof of our main results. Last, in Sections \ref{sec: proof of gamma} and \ref{sec: proof of rho}, we provide the proof of our main results. Throughout  $C,\tilde{C}$ denotes generic constants which are strictly positive and may change from one line to another.

\section{Main results and discussions}\label{sec:main}
In this section we introduce some preliminary tools and state the main results of this article. We denote by  $\mathcal{C}^{\alpha,\frac{\alpha}{2}}(X)$, with $X = \overline{Q}$ or $X=\overline{\Sigma}$, the H\"older space of elements $f \in \mathcal{C}(X) $ satisfying the condition
$$
[f]_{\alpha,\frac{\alpha}{2}}=\sup\left\{\frac{|f(x,t)-f(y,s)|}{( |x-y|^2+|t-s|)^{\frac{\alpha}{2}}}:\  (x,t),\ (y,s)\in X,\ (x,t)\neq(y,s)\right\}<\infty .
$$
Moreover, we define the H\"older space $\mathcal{C}^{2+\alpha,1+\frac{\alpha}{2}}(X)\subset \mathcal{C}([0,T];\mathcal{C}^2(Y))\cap \,\mathcal{C}^1([0,T];\mathcal{C}(Y))$ (where $Y=\overline{\Omega}$ for $X= \overline{Q}$, and $Y=\partial\Omega$ for $X=\overline{\Sigma}$) to be the set composed of functions $f$ satisfying $\partial_x^\beta f ,\partial_t f\in \mathcal{C}^{\alpha,\frac{\alpha}{2}}(X)$ for $\beta\in\mathbb N^n$ with $|\beta| \leq 2$. The space $\mathcal{C}^{2+\alpha,1+\frac{\alpha}{2}}(X)$ is equipped with the standard norm
$$
    \|f\|_{\mathcal{C}^{2+\alpha,1+\frac{\alpha}{2}}(X)} = \sum_{|\beta|+2\nu\leq 2} \| \partial_x^\beta \partial_t^\nu f\|_{\mathcal{C}(X)} + [ \partial_x^\beta \partial_t^\nu f]_{\alpha,\frac{\alpha}{2}}  .
$$
In addition, for an open subset $S$ of $\partial\Omega$, we define the subspace $J_S$ of $\mathcal{C}^{2+\alpha,1+\frac{\alpha}{2}}(\overline{\Sigma})$ to be 
\begin{equation}
    J_S := \{ g \in \mathcal{C}^{2+\alpha,1+\frac{\alpha}{2}}(\overline{\Sigma}) : \supp (g) \subset S\times (0,T), \,  g\mid_{\partial\Omega \times \{0\}} = 0 , \,  \partial_t g\mid_{\partial\Omega \times \{0\}} = 0 \}.
\end{equation}
We will prove in Propositions \ref{prop: existence of forward problem} and \ref{uniqueness} that for $g\in J_S$ sufficiently small, the problem \eqref{IBVP} admits a unique solution $u_{\lambda,g}\in \mathcal{C}^{2+\alpha,1+\frac{\alpha}{2}}(\overline{Q})$. Thus, the DN map \eqref{eqn:defn of N} is well defined in a suitable neighborhood of zero in $J_S$. Using the Sobolev space $H_0^{1/2,1/2}(S\times(0,T))$ defined in \eqref{eqn:defn of H_0^1/2} below and its dual space $H^{-1/2,-1/2}(S\times(0,T))$, we can state the first main result, which is a Lipschitz stability result for the quasilinear term $\gamma$.
\begin{Thm}
\label{thm:est for gamma}
Suppose that for $j=1,2$, there are maps $\gamma^j,\rho^j \in \mathcal{C}^1([0,T];\mathcal{C}^3(\R)) $ satisfying \eqref{positivity condition of gamma,rho}. Let $\mathcal{N}_\lambda = \mathcal{N}_{\lambda,\gamma^1,\rho^1} - \mathcal{N}_{\lambda,\gamma^2,\rho^2}$, with $\mathcal{N}_{\lambda, \gamma^j,\rho^j}$ being partial parabolic DN maps, cf. \eqref{eqn:defn of N}. Then, for all $R>0$, there exists a constant $C=C(\Omega,S,T,A)>0$ such that 
    \begin{equation}
    \label{est1}
        \|\gamma^1-\gamma^2\|_{L^\infty((0,T)\times(-R,R))} \leq C\sup_{\lambda \in [-R,R]} \sup_{g\in B_1} \limsup_{k \to +\infty} \left(k\left\|\mathcal{N}_\lambda\left(\frac{g}{k}\right)\right\|_{H^{-1/2,-1/2}(S\times(0,T))}\right) <\infty ,
    \end{equation}
    with $B_1 = \{g \in  J_S:\  \|g\|_{H^{1/2,1/2}_0(S\times(0,T))} \leq 1\}$.
\end{Thm}

In the second main result, we combine the Lipschitz stable determination of $\gamma$ with a H\"older stability result for the quasilinear term $\rho$.
\begin{Thm}
\label{thm:est of rho}
 Suppose that $A=Id$ and $n\geq 3$, and that for $j=1,2$, there are maps $\gamma^j,\rho^j \in    \mathcal{C}^1([0,T]; \mathcal{C}^3(\R))$ satisfying \eqref{positivity condition of gamma,rho}. Assume also that there exists a strictly positive continuous function $\kappa:\R\to(0,\infty)$ such that
\begin{equation}
 \label{eqn:additional assumption for rho}
    \sup_{t\in [0,T]} \partial_t\rho^j(t,s) \leq \kappa(s),  \;  \quad s\in\R,\ j=1,2, 
\end{equation}
and that for all $\lambda\in\R$ there exists $t_\lambda\in(0,T)$ such that
	\begin{equation}
  \label{cond2} \max_{t\in[0,T]} |\rho^1(t,\lambda)-\rho^2(t,\lambda)|=|\rho^1(t_\lambda,\lambda)-\rho^2(t_\lambda,\lambda)|.\end{equation}
  Fix $\mathcal{N}_\lambda = \mathcal{N}_{\lambda,\gamma^1,\rho^1} - \mathcal{N}_{\lambda,\gamma^2,\rho^2}$, where $\mathcal{N}_{\lambda,\gamma^j,\rho^j}$ are DN maps defined by \eqref{eqn:defn of N}. Then, for all $R>0$, the estimate \eqref{est1} holds, and there exists $C=C(\Omega,S,\kappa,m,R,T)>0$  such that 
    \begin{equation}
    \label{estt}
        \|\rho^1-\rho^2\|_{L^\infty((0,T)\times(-R,R))} \leq C\left(\sup_{\lambda \in [-R,R]} \sup_{g\in B_1} \limsup_{k \to +\infty} \left(k\left\|\mathcal{N}_\lambda\left(\frac{g}{k}\right)\right\|_{H^{-1/2,-1/2}(S\times(0,T))}\right) \right)^{\frac{1}{9}} < \infty.
    \end{equation}
\end{Thm}
\begin{remark}
    There are several \textit{a priori} conditions that one may impose on the unknown parameter $\rho$ to guarantee the condition \eqref{cond2}. For example, one could impose that for all $\lambda\in\R$ there exists $\epsilon_\lambda\in(0,T/4)$ such that $$\partial_t\rho(t,\lambda)=0,\quad t\in (0,\varepsilon_\lambda) \cup (T-\varepsilon_\lambda,T),$$
 or that  $\rho(s,\lambda)$ is known for $s=0,T$ and $\lambda\in\R$.
\end{remark}

To the best of our knowledge, we obtain in Theorems \ref{thm:est for gamma} and \ref{thm:est of rho} the first results of stable determination of quasilinear terms appearing in a parabolic equation from boundary measurements. Our results provide the simultaneous stable determination of the thermal conductivity $\gamma$ and volumetric heat capacity $\rho$. We derive Lipschitz and H\"older stability estimates while most of existing stability estimates for parabolic equations are of logarithmic type. Such stability estimates can have significant applications in numerical reconstruction where stability estimates can be used for proving the convergence rates of numerical schemes \cite{egger2018tikhonov,werner2019convergence}. We mention also that Theorem \ref{thm:est for gamma} is stated for parabolic equations with a general class of space dependent variable coefficients and all our results are stated with measurements of the flux restricted to an arbitrary open subset of the boundary. Even in terms of uniqueness, the simultaneous recovery of these two class of parameters $\gamma$ and $\rho$ seems to be new.

We observe that the stability estimate \eqref{est1} is
an unconditional Lipschitz one in the sense that it is stated without any \textit{a priori} conditions imposed on the unknown parameter $\gamma$, and the constant $C$ appearing in \eqref{est1} is completely independent of the parameters $\gamma$, $\rho$ and $R$. Such stability estimate differs from most existing stability estimates  for the parabolic equations, and in the context of nonlinear inverse problems, we are only aware of the works \cite{kian2023lipschitz,Ki24} for stability estimates of similar type derived for elliptic equations. 

In contrast to the determination of the conductivity $\gamma$, the stability estimate \eqref{estt} is a conditional H\"older one requiring the \textit{a priori} estimate \eqref{cond2}. Moreover, the constant $C$ appearing in \eqref{estt}
depends on the different parameters $\kappa$, $m$ and $R$. This can be compared with the determination of a semilinear term for elliptic equations of
\cite[Theorem 1.2]{kian2023lipschitz} where the dependency with respect to $R$ comes from the fact that the stable determination of such nonlinear functions must be localized. Despite these restrictions, we note that, since the stability estimate \eqref{est1} is independent of the parameter $\rho$, the stability estimate \eqref{estt} can be combined with \eqref{est1} into the simultaneous stable determination of the parameters $\gamma$ and $\rho$.

Our approach for proving both estimates involves first a linearization argument that we consider along with the well-posedness of problem \eqref{IBVP} with sufficiently small Dirichlet boundary data $ g\in J_S$. Then we introduce the partial parabolic DN map $\Lambda_{\lambda,\gamma,\rho}$, associated with the linearized equation, that we define for less regular Dirichlet boundary data in $H_0^{1/2,1/2}(S\times(0,T))$. Last, we use a novel construction of singular solutions to the linearized equation to prove the stability estimates. Namely, we do not utilize the fundamental solution to parabolic equations, but instead we employ a new class of singular solutions that we built from a fundamental solution to an elliptic equation. In view of the flexibility of our construction of singular solutions, we believe that this approach can be applied to similar class of inverse problems for different classes of evolution PDEs, e.g., time-fractional subdiffusion. 
    
\section{Forward problem and associated Dirichlet-to-Neumann map}

\label{Forward problem and DN map}
In this section, we first show that the IBVP \eqref{IBVP} is well-posed for  boundary data lying in $ B_{\varepsilon_\lambda,S} =  \{g \in J_S : \|g\|_{\mathcal{C}^{2+\alpha,1+\frac{\alpha}{2}}(\overline{\Sigma})} < \varepsilon_{\lambda} \} $, with $\varepsilon_\lambda>0$ dependent on $\lambda\in\R$. Then we introduce the associated partial parabolic DN map $\mathcal{N}_{\lambda,\gamma,\rho}$ and show that its Fr\'echet derivative in a certain direction is equal to $\Lambda_{\lambda,\gamma,\rho}$, the partial parabolic DN map associated with the linearized IBVP \eqref{linearised eqn. assoc. w/ DN map}. Last, we define the Sobolev space $H_0^{1/2,1/2}(S\times(0,T))$ and show that both the linear IBVP \eqref{linearised eqn. assoc. w/ DN map} and its associated DN map $\Lambda_{\lambda,\gamma,\rho}$ remain well defined with less regular Dirichlet boundary data $g \in H_0^{1/2,1/2}(S\times(0,T))$.

We start with the existence of solutions of the IBVP \eqref{IBVP}.
    \begin{prop}
    \label{prop: existence of forward problem}
        Suppose $\gamma$, $\rho$ and $A$ satisfy conditions \eqref{ellipticity} and \eqref{positivity condition of gamma,rho}. Then, for all $\lambda \in \R$, there exists $\varepsilon_\lambda= \varepsilon_\lambda(\Omega,T,\lambda,A)>0$ such that for all $g \in B_{\varepsilon_\lambda,S}$, there exists a solution $u_{\lambda,g} \in \mathcal{C}^{2+\alpha,1+\frac{\alpha}{2}}(\overline{Q}) $ to \eqref{IBVP} satisfying 
 $$
            \|u_{\lambda,g} - \lambda\|_{\mathcal{C}^{2+\alpha,1+\frac{\alpha}{2}}(\overline{Q})} \leq C\|g\|_{\mathcal{C}^{2+\alpha,1+\frac{\alpha}{2}}(\overline{\Sigma})} ,
$$
where $C=C(\Omega,T,A)>0$
    \end{prop}
    \begin{proof}
   We follow the approach in \cite[Proposition 2.1]{kian2023lipschitz}. First note that if $v$ solves the IBVP 
       \begin{equation}
       \label{`homogenous' IBVP}
            \begin{cases}
\rho(t,v+\lambda)\partial_tv- \nabla \cdot (\gamma(t,v+\lambda)A(x)  \nabla v)=0  & \text{in } Q ,
\\
v=g &\text{on } \Sigma,\\
v(x,0)=0 & x\in\Omega,
\end{cases}
       \end{equation}
then $u_{\lambda,g}:= v + \lambda$ solves \eqref{IBVP}. So it suffices to show that for $g \in B_{\varepsilon_\lambda,S}$, there exists $v \in \mathcal{C}^{2+\alpha,1+\frac{\alpha}{2}}(\overline{Q})$ that solves \eqref{`homogenous' IBVP} and satisfies
\begin{equation}
\label{ineq. for `homogenous' IBVP}
    \|v\|_{\mathcal{C}^{2+\alpha,1+\frac{\alpha}{2}}(\overline{Q})} \leq C \|g\|_{\mathcal{C}^{2+\alpha,1+\frac{\alpha}{2}}(\overline{\Sigma})} .
\end{equation}
Consider the function spaces
$$
\begin{aligned}
   &  \mathcal{H}_1  = \{h \in \mathcal{C}^{2+\alpha,1+\frac{\alpha}{2}}(\overline{\Sigma}) : h|_{\partial\Omega\times\{0\}} = 0, \, \partial_t h |_{\partial\Omega\times\{0\}} = 0\}, \\
    & \mathcal{H}_2 = \{w\in \mathcal{C}^{2+\alpha,1+\frac{\alpha}{2}}(\overline{Q}) : w|_{\overline{\Omega}\times\{0\}} = 0, \, \partial_t w |_{\partial\Omega\times\{0\}} = 0\}, \\
   & \mathcal{H}_3 = \{F \in \mathcal{C}^{\alpha,\frac{\alpha}{2}}(\overline{Q}) : F\mid_{\partial\Omega \times \{0\}} = 0 \},
\end{aligned}
$$
noting that $J_S \subset \mathcal{H}_1$. We define a mapping $\mathcal{G}:\mathcal{H}_1 \times \mathcal{H}_2 \to \mathcal{H}_3 \times \mathcal{H}_1$ by 
$$
    \mathcal{G}(h,w) = (\rho(t,w+\lambda)\partial_tw- \nabla \cdot (\gamma(t,w+\lambda)A(x)  \nabla w), w|_{\overline{\Sigma}}-h) .
$$
By the  regularity conditions imposed on $\gamma,\rho$ and $A$, the map $\mathcal{G}$ is $\mathcal{C}^1$ in the sense of Fr\'echet derivatives, and moreover for $\tilde{w}\in \mathcal{H}_2$,
$$
    \partial_w\mathcal{G}(0,0)\tilde{w} = (\rho(t,\lambda)\partial_t \tilde{w} - \gamma(t,\lambda) \nabla\cdot (A(x)\nabla\tilde{w}), \tilde{w}|_{\partial\Omega}).
$$
Note that $\mathcal{G}(0,0) = (0,0)$ and we aim to apply the implicit function theorem. To this end, we show that the map $\partial_w\mathcal{G}(0,0):\mathcal{H}_2 \to \mathcal{H}_3 \times \mathcal{H}_1$ is an isomorphism. Consider arbitrary $(F,h) \in  \mathcal{H}_3 \times \mathcal{H}_1$ and the linear IBVP
\begin{equation}
\label{linearised problem}
            \begin{cases}
\rho(t,\lambda)\partial_tw- \gamma(t,\lambda)\nabla \cdot (A(x)  \nabla w)=F  & \text{in } Q ,
\\
w=h &\text{on } \Sigma,\\
w(x,0)=0 & x\in\Omega.
\end{cases}
\end{equation}
By the regularity of $F,h$ and the compatibility conditions on $h$ and \cite[p. 320, Theorem 5.3]{ladyzhenskaia1968linear}, we deduce that the problem \eqref{linearised problem} has a unique solution $w \in \mathcal{H}_2$, and moreover it depends continuously on $(F,h)$. Thus, $\partial_w\mathcal{G}(0,0)$ is indeed an isomorphism, and the implicit function theorem yields the existence of $\varepsilon_\lambda>0$ and a $\mathcal{C}^1$ map $\varphi:B_{\varepsilon_\lambda,S} \to \mathcal{H}_2$,
such that for $g \in B_{\varepsilon_\lambda,S}$, we have $\mathcal{G}(g,\varphi(g)) = (0,0)$ and $\varphi(0)=0$. Then $v=\varphi(g)$ solves \eqref{`homogenous' IBVP} and, since $\varphi$ is $\mathcal{C}^1$ with $\varphi(0)=0$, we may apply the mean value theorem to obtain the inequality \eqref{ineq. for `homogenous' IBVP}. This concludes the proof of the proposition    .
\end{proof}

    The next result gives the uniqueness of solutions of problem \eqref{IBVP}.
\begin{prop}\label{uniqueness}
    Suppose $\gamma$, $\rho$ and $A$ satisfy conditions \eqref{ellipticity} and \eqref{positivity condition of gamma,rho}, and let $\varepsilon_\lambda>0$ be as in the conclusion of Proposition \ref{prop: existence of forward problem}. Then, for any $g \in B_{\varepsilon_\lambda,S}$, the IBVP \eqref{IBVP} has a unique solution $u_{\lambda,g} \in \mathcal{C}^{2+\alpha,1+\frac{\alpha}{2}}(\overline{Q}) $.
\end{prop}
\begin{proof}
   Assume that $u_1,u_2\in \mathcal{C}^{2+\alpha,1+\frac{\alpha}{2}}(\overline{Q})$ are both solutions of \eqref{IBVP}. Then in $Q$, we have 
    $$
        \rho(t,u_1)\partial_tu_1- \nabla \cdot (\gamma(t,u_1)A(x)  \nabla u_1) - \rho(t,u_2)\partial_tu_2 + \nabla \cdot  (\gamma(t,u_2)A(x) \nabla u_2) =0 .
$$
    Let $u = u_1 - u_2$. We first consider the terms containing $\rho$, and denote the rest by $$\zeta:=- \nabla \cdot (\gamma(t,u_1)A(x)  \nabla u_1)  + \nabla \cdot  (\gamma(t,u_2)A(x) \nabla u_2).$$ By the fundamental theorem of calculus, we get 
$$
    \begin{aligned}
        \rho(t,u_1)\partial_t u +  \zeta &= -\rho(t,u_1)\partial_t u_2 + \rho(t,u_2)\partial_t u_2 \\
        &= -\partial_t u_2(u_1-u_2) \int_0^1 \partial\rho(t,su_1 + (1-s)u_2) ds,
    \end{aligned}
$$
where $\partial\rho$ indicates the partial derivative of $\rho$ in the second argument. Then we obtain 
$$
    \rho(t,u_1) \partial_t u + qu + \zeta = 0 ,
$$
with $q := (\partial_t u_2) \int_0^1 \partial\rho(t,su_1 + (1-s)u_2) ds \in \mathcal{C}^{\alpha,\frac{\alpha}{2}}(\overline{Q})$. Next, let  $\xi := \rho(t,u_1) \partial_t u + qu$. Similarly, by adding and subtracting $\nabla\cdot(\gamma(t,u_1)A(x)\nabla u_2)$, we have
\begin{align*}
    &\quad -\nabla \cdot (\gamma(t,u_1)A(x)  \nabla u) + \xi\\
    & = \nabla \cdot (\gamma(t,u_1)A(x) \nabla u_2) - \nabla \cdot (\gamma(t,u_2)A(x) \nabla u_2) 
    = \partial_{x_i} ([\gamma(t,u_1)-\gamma(t,u_2)]a_{ij}\partial_{x_j}u_2) \\
    &= (\gamma(t,u_1)-\gamma(t,u_2))\partial_{x_i}(a_{ij}\partial_{x_j}u_2) + \partial_{x_i}(\gamma(t,u_1)-\gamma(t,u_2))a_{ij}\partial_{x_j} u_2
    =: B + \tilde{B},
\end{align*}
where repeated indices indicate summation. By the fundamental theorem of calculus, we have
$$
    B = \partial_{x_i}(a_{ij}\partial_{x_j}u_2) (u_1-u_2) \int_0^1 \partial\gamma(t,su_1+(1-s)u_2) ds 
    ,
$$
with $\partial \gamma$ indicating the partial derivative of $\gamma$ in the second argument. Analogously, we have
$$
    \tilde{B} = a_{ij}\partial_{x_j}u_2\partial_{x_i}\left( (u_1-u_2) \int_0^1 \partial\gamma(t,su_1+(1-s)u_2) ds \right)   .
$$
Thus, $B+ \tilde{B}$ takes the form $-\mathcal{B}\cdot \nabla u - \tilde{q}u$ with $\mathcal{B},\tilde{q} \in \mathcal{C}^{\alpha,\frac{\alpha}{2}}(\overline{Q})$, and in $Q$, $u$ satisfies
$$
    \rho(t,u_1)\partial_t u - \nabla \cdot (\gamma(t,u_1)A(x) \nabla u) + \mathcal{B}\cdot \nabla u + (q+\tilde{q})u = 0.
$$
Considering also the boundary conditions of $u_1$ and $u_2$, it follows $u\in \mathcal{C}^{2+\alpha,1+\frac{\alpha}{2}}(\overline{Q})$ solves
$$
    \begin{cases}
 \rho(t,u_1)\partial_t u - \nabla \cdot ( \gamma(t,u_1)A(x)\nabla u) + \mathcal{B}\cdot \nabla u + (q+\tilde{q})u = 0  & \text{in } Q ,
\\
u=0 &\text{on } \Sigma,\\
u(x,0)=0 & x\in\Omega.
\end{cases}
$$
By \cite[p. 320, Theorem 5.3]{ladyzhenskaia1968linear}, $u \equiv 0$ is the unique solution to this IBVP, and the result follows. 
\end{proof}

Using Propositions \ref{prop: existence of forward problem} and \ref{uniqueness}, we can define the partial parabolic DN map $\mathcal{N}_{\lambda,\gamma,\rho}$ associated with the IBVP \eqref{IBVP}. We also define the partial parabolic DN map $\Lambda_{\lambda,\gamma,\rho}$ associated with the linearized problem and we connect both maps.
\begin{definition}
    For any $\lambda \in \R$ and for any open subset $S \subset \partial \Omega$, we define the partial parabolic DN map associated with the IBVP \eqref{IBVP} to be the map $\mathcal{N}_{\lambda,\gamma,\rho}:B_{\varepsilon_\lambda,S} \to \mathcal{C}(\overline{S} \times [0,T])$: 
    \begin{equation}
    \label{Definition of DN map for nonlinear}
        \mathcal{N}_{\lambda,\gamma,\rho}g  = \gamma(t,u_{\lambda,g}) (A(x)\nabla u_{\lambda,g} \cdot \nu(x))|_{S\times(0,T)} ,
    \end{equation}
    where $\varepsilon_\lambda$ is as in Proposition \ref{prop: existence of forward problem}, 
    and $u_{\lambda,g}\in \mathcal{C}^{2+\alpha,1+\frac{\alpha}{2}}(\overline{Q})$ is the solution of \eqref{IBVP}. 
\end{definition}
\begin{definition}
  For $\lambda \in \R$ and $S \subset \partial\Omega$, and for $g \in J_S$, consider the linear IBVP
\begin{equation}
\label{linearised eqn. assoc. w/ DN map}
            \begin{cases}
\rho(t,\lambda)\partial_tw- \gamma(t,\lambda)\nabla \cdot (A(x)  \nabla w)=0  & \text{in } Q ,
\\
w=g &\text{on } \Sigma,\\
w(x,0)=0 & x\in\Omega.
\end{cases}
\end{equation}
We denote the unique solution of \eqref{linearised eqn. assoc. w/ DN map} by $w_{\lambda,g}$, and define the partial parabolic DN map $\Lambda_{\lambda,\gamma,\rho}: J_S \to L^2(S\times(0,T))$ associated with \eqref{linearised eqn. assoc. w/ DN map} to be a bounded linear map given by
\begin{equation}
\label{eqn: defn of linear dn}
    \Lambda_{\lambda,\gamma,\rho}g = \gamma(t,\lambda) (A(x)\nabla w_{\lambda,g} \cdot \nu(x))|_{S\times(0,T)} .
\end{equation}
\end{definition} 
\begin{prop}
\label{prop: derivative of N is Lambda}
    For any $g \in J_S$, there holds 
    \begin{equation}
    \label{derivative of N is Lambda}
    \partial_s \mathcal{N}_{\lambda,\gamma,\rho}( sg) \mid_{s=0} =  \Lambda_{\lambda,\gamma,\rho}g ,
    \end{equation}
		 with $\partial_s$ denoting the Fr\'echet derivative of maps from $J_S$ to $L^2(S\times(0,T))$.
\end{prop}
\begin{proof}
By Proposition \ref{prop: existence of forward problem}, for any $\lambda \in \R$ and $g\in J_S$, the IBVP \eqref{IBVP} with boundary data $\lambda+sg$ has a unique solution $u_{\lambda,sg}\in \mathcal{C}^{2+\alpha,1+\frac{\alpha}{2}}(\overline{Q})$ for $s \in [-s_g,s_g]$ (with $s_g= \varepsilon_{\lambda}/({\|g\|_{\mathcal{C}^{2+\alpha,1+\frac{\alpha}{2}}(\overline{\Sigma})}+1})$). Hence,  $\mathcal{N}_{\lambda,\gamma,\rho}( sg)$ is well defined for $s\in[-s_g,s_g]$, with 
$$\mathcal{N}_{\lambda,\gamma,\rho}( sg)=\gamma(t,\lambda+\varphi(sg)) (A(x)\nabla \varphi(sg) \cdot \nu(x))|_{S\times(0,T)},$$
where $\varphi$ is the $\mathcal{C}^1$ mapping from $B_{\varepsilon_\lambda,S}$ to $\mathcal{H}_2 \subset \mathcal{C}^{2+\alpha,1+\frac{\alpha}{2}}(\overline{Q})$ defined in the proof of Proposition \ref{prop: existence of forward problem}. Thus the map $s\mapsto\mathcal{N}_{\lambda,\gamma,\rho}( sg)$ is $\mathcal{C}^1$ from $[-s_g,s_g]$ to $L^2(\Sigma)$ and the expression $ \partial_s \mathcal{N}_{\lambda,\gamma,\rho}( sg) \mid_{s=0}$ in \eqref{derivative of N is Lambda} is well defined. Next, note that $u_{\lambda,0} = \lambda$, since $u\equiv \lambda$ is a solution to \eqref{IBVP} when $g \equiv 0$. We compute $\partial_su_{\lambda,0}:=\partial_su_{\lambda,sg}\mid_{s=0}$. Differentiating  \eqref{IBVP} with respect to $s$ gives 
$$
\partial_s[\rho(t,u_{\lambda,sg})\partial_tu_{\lambda,sg}- \gamma(t,u_{\lambda,sg})\nabla \cdot (A(x)  \nabla u_{\lambda,sg})]_{s=0}= 0 .
$$
We now compute the Fr\'echet derivative $\partial_s[\rho(t,u_{\lambda,sg})\partial_tu_{\lambda,sg}]_{s=0}$. First, we see
  $$
         \partial_s[\rho(t,u_{\lambda,sg})\partial_tu_{\lambda,sg}]_{s=0} = \rho(t,u_{\lambda,0})\partial_s[\partial_t u_{\lambda,sg}]_{s=0} + \partial_s[\rho(t,u_{\lambda,sg})]_{s=0}\partial_t u_{\lambda,0}
   $$
     In the term $\rho(t,u_{\lambda,0})\partial_s[\partial_t u_{\lambda,sg}]_{s=0} $, we may interchange the $s,t$ derivatives due to the assumed regularity of $u_{\lambda,sg}$. Moreover, substituting $u_{\lambda,0} = \lambda$ yields that $\partial_s[\rho(t,u_{\lambda,sg})]_{s=0}\partial_t u_{\lambda,0}=0$ and hence 
     $$
          \partial_s[\rho(t,u_{\lambda,sg})\partial_tu_{\lambda,sg}]_{s=0}  = \rho(t,\lambda) \partial_t \partial_su_{\lambda,0} .
   $$
Similarly, we get
   $$
         \partial_s[- \nabla \cdot (\gamma(t,u_{\lambda,sg})A(x)  \nabla u_{\lambda,sg})]_{s=0} = -\gamma(t,\lambda)\nabla\cdot(A(x)\nabla\partial_su_{\lambda,0}).
    $$
     Moreover, since $u_{\lambda,sg} = \lambda + sg$ on $\Sigma$ and $u_{\lambda,sg} (x,0) = 0$ for $x \in \Omega$, we have $\partial_s u_{\lambda,0} = g$ on $\Sigma$ and $\partial_su_{\lambda,0}(x,0) = 0$ for $ x \in \Omega$. Combining the preceding identities implies that  $\partial_su_{\lambda,0}$ solves 
$$
\begin{cases}
    \rho(t,\lambda) \partial_t\partial_su_{\lambda,0} - \gamma(t,\lambda)\nabla \cdot(A(x) \nabla \partial_su_{\lambda,0}) = 0 & \text{in } Q,\\
    \partial_su_{\lambda,0} = g & \text{on } \Sigma,\\
    \partial_su_{\lambda,0}(x,0) = 0 & x\in\Omega.
\end{cases}
$$
Hence, $\partial_su_{\lambda,0}$ is a solution to \eqref{linearised eqn. assoc. w/ DN map}, and so $\partial_su_{\lambda,0} \equiv w_{\lambda,g} $. Last, we prove \eqref{derivative of N is Lambda} by direct computation:
$$
     \begin{aligned}
       &\quad \partial_s \mathcal{N}_{\lambda,\gamma,\rho}( sg) \mid_{s=0} = \partial_s[\gamma(t,u_{\lambda,sg})A(x)\nabla u_{\lambda,sg} \cdot \nu(x)]_{s=0} \\
       &= \partial_s[\gamma(t,u_{\lambda,sg})]_{s=0} A(x)\nabla u_{\lambda,0} \cdot \nu(x) + \gamma(t,u_{\lambda,0}) \partial_s[A(x)\nabla u_{\lambda,sg}\cdot \nu(x)]_{s=0} \\
       &= \partial_s[\gamma(t,u_{\lambda,sg})]_{s=0}A(x)\nabla \lambda \cdot \nu(x)+\gamma(t,\lambda) A(x)\nabla \partial_s u_{\lambda,0}\cdot \nu(x) \\
       &= \gamma(t,\lambda)A(x) \nabla w_{\lambda,g} \cdot \nu(x) 
   =  \Lambda_{\lambda,\gamma,\rho}(g) .
      \end{aligned}
$$
This completes the proof of the proposition.
\end{proof}

Last, we recall the Sobolev space $H_0^{1/2,1/2}(S\times(0,T))$ and prove that the linear IBVP \eqref{linearised eqn. assoc. w/ DN map} admits a unique solution with less regular Dirichlet boundary data $g \in H_0^{1/2,1/2}(S\times(0,T))$.

\begin{definition}
 Following \cite[Chapter 4, Section 2]{LM2}, for $\tau=0,T$, we define ${}_{\tau}H^{1/2,1/2}(\Sigma)$ to be the subspace of $H^{1/2,1/2}(\Sigma) := L^2(0,T;H^{1/2}(\partial\Omega)) \cap H^{1/2}(0,T;L^2(\partial\Omega))$ with elements $g\in H^{1/2,1/2}(\Sigma)$ satisfying the global compatibility conditions $|t-\tau|^{-\frac{1}{2}}g \in L^2(\Sigma)$, i.e. 
 \begin{equation}
     {}_{\tau}H^{1/2,1/2}(\Sigma) = \{ g \in H^{1/2,1/2}(\Sigma) : |t-\tau|^{-\frac{1}{2}}g \in L^2(\Sigma) \} .
 \end{equation}
 For the space ${}_\tau H^{1/2,1/2}(\Sigma)$, we  equip the norm 
\begin{equation}
    \|g\|_{{}_{\tau}H^{1/2,1/2}(\Sigma)} = (\|g\|_{H^{1/2,1/2}(\Sigma)}^2 + \||t-\tau|^{-\frac{1}{2}}g\|_{L^2(\Sigma)}^2)^{1/2} ,\quad g\in {}_{\tau}H^{1/2,1/2}(\Sigma),
\end{equation}
and the associated scalar product. For any open subset $S \subset \partial \Omega$ we  define the space 
 \begin{equation}
 \label{eqn:defn of H_0^1/2}
     H_0^{1/2,1/2} (S \times (0,T)) = \{ g \in H^{1/2,1/2}(\Sigma) : \mathrm{supp}(g) \subset S \times (0,T) \} .
 \end{equation}
By noting that $H_0^{1/2,1/2} (S \times (0,T))\subset{}_{0}H^{1/2,1/2}(\Sigma)\cap {}_{T}H^{1/2,1/2}(\Sigma)$, we equip it with the norm
$\|g\|_{H_0^{1/2,1/2} (S \times (0,T))}=(\|g\|_{{}_{0}H^{1/2,1/2}(\Sigma)}^2+\|g\|_{{}_{T}H^{1/2,1/2}(\Sigma)}^2)^{1/2}$ for $g\in H_0^{1/2,1/2} (S \times (0,T))$, and the associated scalar product. Last, we define the spaces  ${}_\tau H^1(Q)$ and $H^{-1/2,-1/2}(S\times (0,T))$, with 
\begin{equation}
  {}_{\tau}H^{1}(Q) = \{v\in H^1(Q):\ v|_{t=\tau}= 0\} , 
\end{equation}
and with $H^{-1/2,-1/2}(S\times (0,T))$ being the dual space of $H_0^{1/2,1/2} (S \times (0,T))$.

\end{definition}
We first prove a lifting result of $H_0^{1/2,1/2}(S\times(0,T))$; see also \cite[Chapter 4, Section 2]{LM2}.
\begin{lem}
\label{lem:lifting}
    Let $\tau=0,T$. The trace map $v\mapsto v|_{\Sigma}$ is a continuous and surjective linear map from  ${}_{\tau}H^{1}(Q)$ to ${}_{\tau}H^{1/2,1/2}(\Sigma)$. Hence we can define a continuous linear map $E_\tau$ from ${}_{\tau}H^{1/2,1/2}(\Sigma)$ to ${}_{\tau}H^{1}(Q)$ such that 
		\begin{equation}\label{ET}E_\tau g|_{\Sigma}=g,\quad g\in {}_{\tau}H^{1/2,1/2}(\Sigma).\end{equation}
\end{lem}
\begin{proof}
We prove only the case $\tau=0$, as the case $\tau=T$ is similar. By \cite[Chapter 4, Theorems 2.1 and 2.2]{LM2}, the trace map $v\mapsto v|_{\Sigma}$ is continuous from  ${}_{0}H^{1}(Q)$ to ${}_{0}H^{1/2,1/2}(\Sigma)$ and hence for the first assertion, it suffices to show the surjectivity of the map. To this end, fix $g\in {}_{0}H^{1/2,1/2}(\Sigma)$. Since ${}_{0}H^{1/2,1/2}(\Sigma)\subset H^{1/2,1/2}(\Sigma)$ and by applying \cite[Chapter 1, Theorem 4.2]{LM1} (see also  \cite[Chapter 4, Theorem 2.1]{LM2}), we deduce that there exists $w\in H^1(Q)$ such that $w|_{\Sigma}=g$. Consider $w_0=w|_{t=0}$. By 
\cite[Chapter 4, Theorem 2.1]{LM2}, we have $w_0\in H^{1/2}(\Omega)$ and, \cite[Chapter 4, Theorem 2.2]{LM2} and the fact $t^{-1/2}g\in L^2(\Sigma)$, imply $\int_{\Omega}\textrm{dist}(x,\partial\Omega)^{-1}|w_0(x)|^2dx<\infty.$
Therefore, in light of \cite[ Chapter 1, Theorem 11.7]{LM1}, $w_0$ lies in the interpolation space of order $1/2$ between $H^1_0(\Omega)$ and $L^2(\Omega)$ (cf. \cite[Chapter 1, Definition 2.1]{LM1}). Then, by \cite[Chapter 1, Theorem 3.2]{LM1}, we can find $v\in H^1(Q)\cap L^2(0,T;H^1_0(\Omega))$ such that $v|_{t=0}=w_0$. Thus, we have $w-v\in {}_{0}H^{1}(Q)$ and, since $v\in  L^2(0,T;H^1_0(\Omega))$, we find  $(w-v)|_{\Sigma}=w|_{\Sigma}-v|_{\Sigma}=g$. This proves the surjectivity of the trace map, and the right inverse $E_0$ of the trace map between  ${}_{0}H^{1}(Q)$ and ${}_{0}H^{1/2,1/2}(\Sigma)$ fulfills \eqref{ET}.
\end{proof}

\begin{lem}
\label{lem:existence of solution with less regular boundary data}
   Suppose that $\gamma$, $\rho$ and $A$ satisfies \eqref{ellipticity} and \eqref{positivity condition of gamma,rho}, and $g \in H_0^{1/2,1/2}(S\times(0,T))$. Then the linear IBVP \eqref{linearised eqn. assoc. w/ DN map} still admits an unique solution  $w_{\lambda, g} \in H^1(0,T ; H^{-1} (\Omega)) \cap L^2(0,T ; H^1(\Omega))$, and there exists $C=C\left(\Omega, T, A, m(\lambda),\sup_{t\in(0,T)}\frac{\rho(t,\lambda)}{\gamma(t,\lambda)} \right)$ such that 
    \begin{equation}
    \label{eqn:estimate for solution w/ less regular boundary data}
        \|w_{\lambda,g}\|_{H^1(0,T;H^{-1}(\Omega))} + \|w_{\lambda,g}\|_{L^2(0,T;H^1(\Omega))} \leq C\|g\|_{H^{1/2,1/2}_0(S\times(0,T))}.
    \end{equation}
\end{lem}
\begin{proof}
    First, since $g \in H_0^{1/2,1/2}(S\times(0,T))$, we can define $w=E_0g\in {}_{0}H^{1}(Q)$ as in Lemma \ref{lem:lifting} with $w\mid_\Sigma = g$ and $\|w\|_{H^1(Q)} \leq C\|g\|_{H_0^{1/2,1/2}(S\times (0,T))}$. We decompose $w_{\lambda,g}$ by $ w_{\lambda,g}= w+\xi$, where $\xi$ solves the IBVP
    \begin{equation}
            \begin{cases}
            \label{BVP for xi}
        \rho_\lambda(t)\partial_t\xi - \gamma_\lambda(t) \nabla \cdot (A(x)\nabla \xi )= -\rho_\lambda(t) \partial_t w + \gamma_\lambda(t)\nabla \cdot(A(x) \nabla w)  & \text{in } Q, \\
        \xi = 0 & \text{on }\Sigma, \\
        \xi(x,0) = 0 & x\in \Omega ,
    \end{cases}
    \end{equation}
	and $\rho_\lambda$ and $\gamma_\lambda$ denote the maps $t\mapsto\rho(t,\lambda)$ and $t\mapsto\gamma(t,\lambda)$. By \cite[Chapter 3, Theorem 4.1]{LM1}, since $-\rho_\lambda\partial_tw + \gamma_\lambda\nabla \cdot (A \nabla w) \in L^2(0,T;H^{-1}(\Omega))$, there exists a unique solution $\xi \in H^1(0,T; H^{-1}(\Omega)) \cap L^2(0,T;H_0^1(\Omega))$ to \eqref{BVP for xi}. By the same Theorem, $\xi$ satisfies the estimate 
  \begin{equation}
  \label{eqn:est for xi in well-posedness for less regular bndry data}
    \|\xi\|_{H^1(0,T;H^{-1}(\Omega))} + \|\xi\|_{L^2(0,T;H^1(\Omega))} \leq C \|\rho_\lambda\partial_tw - \gamma_\lambda\nabla \cdot (A \nabla w)\|_{L^2(0,T;H^{-1}(\Omega))} . 
  \end{equation}
Hence there exists an unique solution $w_{\lambda,g}=w+\xi$ to \eqref{linearised eqn. assoc. w/ DN map} lying in $H^1(0,T;H^{-1}(\Omega))\cap L^2(0,T;H^1(\Omega))$.  Last, using the estimate \eqref{eqn:est for xi in well-posedness for less regular bndry data}, we get
\begin{align*}
      &\quad \|w_{\lambda,g}\|_{H^1(0,T;H^{-1}(\Omega))} + \|w_{\lambda,g}\|_{L^2(0,T;H^1(\Omega))} \\
      &= \|w+\xi\|_{H^1(0,T;H^{-1}(\Omega))} + \|w+\xi\|_{L^2(0,T;H^1(\Omega))} \\
      & \leq C(\|w\|_{H^1(Q)} +  \|\rho_\lambda\partial_tw - \gamma_\lambda\nabla \cdot (A \nabla w)\|_{L^2(0,T;H^{-1}(\Omega))}) \\
      & \leq C\|w\|_{H^1(Q)} \leq C 
      \|g\|_{H_0^{1/2,1/2}(S\times(0,T))} .
\end{align*}
This completes the proof of the lemma.
\end{proof}
\begin{prop}
\label{prop:extension of DN map and continuity}
    For any $\lambda \in \R$ and $S \subset \partial\Omega$, the DN map $\Lambda_{\lambda,\gamma,\rho}$ can be extended by density to a bounded linear map from $H_0^{1/2,1/2}(S\times (0,T))$ to $H^{-1/2,-1/2}(S \times (0,T))$, defined by 
    \begin{equation}
    \label{defn of DN map for less smooth functions}
    \begin{aligned}
        \langle &\Lambda_{\lambda,\gamma,\rho} g, h\rangle_{H^{-1/2,-1/2}(S\times(0,T)),H^{1/2,1/2}_0(S\times(0,T))} \\  &= \int_Q \left(-\partial_t \rho_\lambda(t) w_{\lambda,g} E_Th - \rho_\lambda(t) w_{\lambda,g} \partial_t (E_Th)  + \gamma_\lambda(t)A(x) \nabla w_{\lambda,g}\cdot \nabla (E_Th) \right)dx dt ,
        \end{aligned}
    \end{equation}
    for $g,h \in H_0^{1/2,1/2}(S\times (0,T))$ and $E_T$ the lifting operator of Lemma \ref{lem:lifting}. Moreover, the mapping $\lambda \mapsto \Lambda_{\lambda,\gamma,\rho}$ is continuous from $\mathbb R$ to $\mathcal B(H_0^{1/2,1/2}(S\times (0,T));H^{-1/2,-1/2}(S \times (0,T)))$.
\end{prop}
\begin{proof}
First, we fix $g \in J_S$ and $h \in H_0^{1/2,1/2}(S\times (0,T))$, and note that $\Lambda_{\lambda,\gamma,\rho}g$ is defined by \eqref{eqn: defn of linear dn}. Also recall that $w_{\lambda,g}\in \mathcal{C}^{2+\alpha,1+\frac{\alpha}{2}}(\overline{Q})$ satisfies $w_{\lambda,g}|_{t=0}=0$, and that $E_Th\in H^1(Q)$ with $E_Th|_{t=T}=0$. Then, we derive
\begin{align}
    \langle \Lambda_{\lambda,\gamma,\rho} g, h\rangle &= \int_\Sigma \gamma_\lambda(t) (A(x) \nabla w_{\lambda,g} \cdot \nu(x))h   d\sigma(x) dt \nonumber\\
    &= \int_Q \gamma_\lambda(t) ( \nabla\cdot(A(x) \nabla w_{\lambda,g})E_Th  +A(x)\nabla w_{\lambda,g} \cdot \nabla (E_Th)) dx dt \nonumber\\
    &= \int_Q \rho_\lambda(t) \partial_tw_{\lambda,g}E_Th + \gamma_\lambda(t) A(x) \nabla w_{\lambda,g} \cdot \nabla (E_Th)dx dt \nonumber\\
    &= \int_Q -\partial_t \rho_\lambda(t) w_{\lambda,g}E_Th - \rho_\lambda(t) w_{\lambda,g} \partial_t (E_Th)dx dt\nonumber  \\ 
    & \ \ \   +\int_Q \gamma_\lambda(t)A(x) \nabla w_{\lambda,g}\cdot \nabla (E_Th) dx dt .     \label{defn of DN map for less smooth functions 2}
\end{align}
Note that for $g \in H_0^{1/2,1/2}(S\times(0,T))$, by Lemma \ref{lem:existence of solution with less regular boundary data}, $w_{\lambda,g} \in H^1(0,T ; H^{-1} (\Omega)) \cap L^2(0,T ; H^1(\Omega))$ and the expression in the last line of \eqref{defn of DN map for less smooth functions 2} is well defined for such class of functions. Hence, by density of  $J_S$ in $H_0^{1/2,1/2}(S \times (0,T))$, one can extend $\Lambda_{\lambda,\gamma,\rho}$ using the identity \eqref{defn of DN map for less smooth functions} as an operator from $H_0^{1/2,1/2}(S\times(0,T))$ to $H^{-1/2,-1/2}(S\times(0,T))$. To show the boundedness of $\Lambda_{\lambda,\gamma,\rho}$, by the inequality \eqref{eqn:estimate for solution w/ less regular boundary data} in Lemma \ref{lem:existence of solution with less regular boundary data}, we conclude 
$$
\begin{aligned}
    |\langle \Lambda_{\lambda,\gamma,\rho}g,h \rangle | &\leq C\|w_{\lambda,g}\|_{L^2(0,T;H^1(\Omega))}\|E_Th\|_{H^1(Q)} \\
    &\leq C\|g\|_{H_0^{1/2,1/2}(S\times(0,T))}\|h\|_{H_0^{1/2,1/2}(S\times(0,T))} .
    \end{aligned}
$$
To show the last assertion, it suffices to show 
\begin{equation}
\label{eqn:main limit to show lambda cts.}
    \lim_{\delta \to 0} \sup_g \|(\Lambda_{\lambda+\delta,\gamma,\rho}- \Lambda_{\lambda,\gamma,\rho})g\|_{H^{-1/2,-1/2}(S\times(0,T))} = 0  ,
\end{equation}
where the supremum is taken over all $g \in H_0^{1/2,1/2}(S\times(0,T))$ with $\|g\|_{H_0^{1/2,1/2}(S \times (0,T))}  = 1$. Using \eqref{defn of DN map for less smooth functions}, we have for arbitrary $g,h \in H_0^{1/2,1/2}(S\times(0,T))$ that 
$$
\begin{aligned}
    |\langle (\Lambda_{\lambda+\delta,\gamma,\rho}-\Lambda_{\lambda,\gamma,\rho})g,h\rangle| = \left\lvert \int_Q \left(\mathcal{A} + \mathcal{B} + \mathcal{C} \right) dx dt\, \right\lvert ,
\end{aligned}
$$
with
$\mathcal{A} = (\partial_t \rho_{\lambda+\delta}) w_{\lambda+\delta,g} E_Th -  (\partial_t \rho_{\lambda} )w_{\lambda,g} E_Th$, $ \mathcal{B} = \rho_{\lambda+\delta} w_{\lambda+\delta,g} \partial_t(E_Th) - \rho_{\lambda} w_{\lambda,g} \partial_t(E_Th)$,
and  $\mathcal{C} = \gamma_{\lambda+\delta} A\nabla w_{\lambda+\delta,g} \cdot \nabla (E_Th) -  \gamma_{\lambda} A\nabla w_{\lambda,g} \cdot \nabla (E_Th)$.
To bound the term $\mathcal{A}$, we rewrite $\mathcal{A}$ as
$$
\mathcal{A} =( E_Th)w_{\lambda+\delta,g}\partial_t(\rho_{\lambda+\delta}-\rho_\lambda) + (E_Th)(\partial_t\rho_\lambda)(w_{\lambda+\delta,g}-w_{\lambda,g}).
$$
The integral of the term $( E_Th)w_{\lambda+\delta,g}\partial_t(\rho_{\lambda+\delta}-\rho_\lambda)$ converges to zero as $\delta \to 0$ by the smoothness of $\rho$. For the second term $(E_Th)(\partial_t\rho_\lambda)(w_{\lambda+\delta,g}-w_{\lambda,g})$, we deduce that $w:=w_{\lambda+\delta,g}-w_{\lambda,g}$ satisfies
$$
    \begin{cases}
        \rho_{\lambda+\delta} \partial_t w - \gamma_{\lambda+\delta} \nabla \cdot (A\nabla w) = (-\rho_{\lambda+\delta} +\rho_\lambda) \partial_t w_{\lambda,g} + (\gamma_{\lambda+\delta}-\gamma_\lambda) \nabla \cdot (A\nabla w_{\lambda,g}) & \text{in }Q,\\
        w = 0 & \text{on }\Sigma,\\
        w(x,0) = 0 & x \in \Omega.
    \end{cases}
$$
Then, again by \cite[Chapter 3, Theorem 4.1]{LM1}, we have
$$
\begin{aligned}
    &\|w\|_{H^1(0,T;H^{-1}(\Omega))}  + \|w\|_{L^2(0,T;H^1(\Omega))} \\ & \quad \quad \leq C \|(\rho_{\lambda+\delta} -\rho_\lambda) \partial_t w_{\lambda,g} - (\gamma_{\lambda+\delta}-\gamma_\lambda) \nabla \cdot (A\nabla w_{\lambda,g})\|_{L^2(0,T;H^{-1}(\Omega))} , 
    \end{aligned}
$$
and the right hand side converges to zero as $\delta \to 0$ by the smoothness of $\gamma,\rho$. The terms $\mathcal{B}$ and $\mathcal{C}$ can be estimated similarly to justify \eqref{eqn:main limit to show lambda cts.} and thus the proof is concluded.
\end{proof}

\section{Construction of singular solutions to the linearized Problem}
\label{section: construction of special solutions}
In this section, we construct a class of singular solutions of the linearized equation \eqref{linearised eqn. assoc. w/ DN map} concentrating at a prescribed point $(x_0,t_0)\in S\times(0,T)$, which will play a fundamental role in the stable determination of the nonlinear term $\gamma$. In contrast to related works on boundary determination of nonlinear terms for parabolic equations (see e.g \cite{Is1,Is2}), we do not employ fundamental solutions of parabolic equations.  Instead, we extend the construction of singular solutions for elliptic equations in \cite[Section 3]{kian2023lipschitz} to the parabolic case. The approach differs also from other constructions of solutions for boundary determination of coefficients of parabolic equations such as the one introduced in \cite{FA}.

Consider an open bounded extension $\Omega_*$ of the domain $\Omega$ with a smooth boundary such that $\overline{\Omega} \subset  \Omega_*$ and we extend the map $A$ into a map defined on $\overline{\Omega_*}$ still denoted by $A$ and lying in  $ \mathcal{C}^3(\overline{\Omega_*};\R^{n\times n})$. We denote by $D$ the diagonal of $\Omega_* \times \Omega_*$, i.e. $D = \{ (x,x) : x \in \Omega_*\}$. By \cite[Theorem 3]{kalf1992ee}, for the uniformly elliptic operator $\mathcal{L}$ given by $\mathcal{L}u = -\nabla \cdot (A(x) \nabla u)$, there exists a parametrix  $P: (\Omega_*\times \Omega_*) \setminus D \to \R $ such that $P \in \mathcal{C}((\Omega_* \times \Omega_*) \setminus D)$, $P(\cdot,y) \in \mathcal{C}^2(\Omega_* \setminus \{y\})$ for each $y \in \Omega_*$, and $\mathcal{L}_xP(x, y) = 0$ for each $x\in \Omega_* \setminus \{y\}$. Moreover, we can split the parametrix $P$ into two terms $H + \mathcal{R}$, where $H$ is the fundamental solution to $\mathcal{L}$ given by 
$$
     H(x,y) = \Psi\left(\sqrt{A(y)(x-y)\cdot(x-y)}\right) ,
$$
with $\Psi$ being the fundamental solution to the Laplace equation. It is well known (see \cite[p.258]{kalf1992ee}) that for $n\geq 3$ and $(x,y) \in (\Omega_* \times \Omega_*) \setminus D$, $H$ satisfies the estimates
 \begin{equation}
 \label{reference estimate for H}
      c|x-y|^{2-n} \leq  |H(x,y)| \leq C|x-y|^{2-n},  \, c |x-y|^{1-n} \leq |\nabla_xH(x,y)| \leq C|x-y|^{1-n} ,
 \end{equation}
 where the constants $c,C>0$ depend only on $\Omega $ and $A$. When $n=2$, we have instead $c|\ln|x-y|| \leq |H(x,y)|\leq C|\ln|x-y||$ and $c|x-y|^{-1} \leq |H(x,y)| \leq C|x-y|^{-1}$. For $n \geq 2$, the remainder term  $\mathcal{R}$ satisfies 
 \begin{equation}
 \label{reference estimate for R}
     c|x-y|^{\frac{5}{2}-n} \leq |\mathcal{R}(x,y)| \leq C|x-y|^{\frac{5}{2}-n} , \, c |x-y|^{\frac{3}{2}-n} \leq |\nabla_x\mathcal{R}(x,y)| \leq C|x-y|^{\frac{3}{2}-n} ,
 \end{equation}
where $(x,y) \in (\Omega_* \times \Omega_*) \setminus D$ and $c,C>0$ are as above. Now, fixing $x_0 \in S$ and by using boundary normal coordinates (see \cite[Theorem 2.12]{kachalov2001inverse} for more details), we can set $\delta'>0$ sufficiently small such that for all $\tau \in (0,\delta')$, there exists unique $y_\tau \in \Omega_*\setminus \overline{\Omega}$ such that $\text{dist}(y_\tau,\partial\Omega) = |y_\tau - x_0| = \tau$. 
We fix $\Omega'$ a bounded domain with $\mathcal{C}^2$ boundary such that $\Omega \varsubsetneq \Omega' \subset \overline{\Omega'} \subset \Omega_*$, and also require that $(\partial\Omega \setminus S )\subset \partial\Omega'$ (see Figure 1). Then, by picking $\delta = \min ( \text{dist}(x_0,\partial \Omega'), 1, \delta')$, we have that for all $\tau \in (0,\delta)$, \text{dist}($y_\tau,\partial \Omega') \geq \delta$. 

\begin{figure}[hbt!]
   \label{fig: construction of singular}
       \centering     
   \includegraphics[scale=0.25]{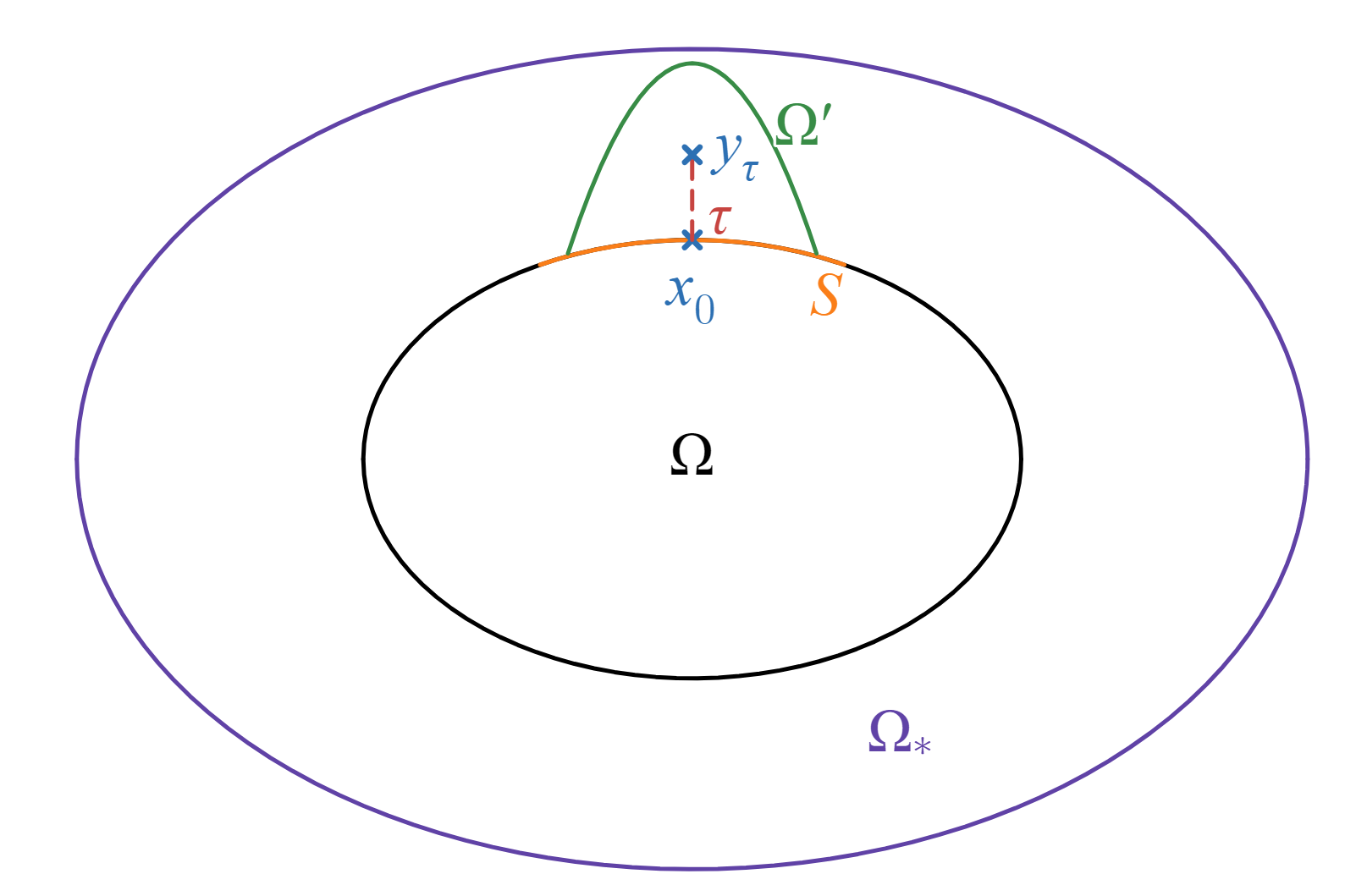}
   \caption{Example of the domains $\Omega,\Omega_*,\Omega'$ and the associated points $x_0,y_\tau$.}
    \end{figure}
For $\tau \in (0,\delta)$, we denote by $v_\tau \in H^2(\Omega')$ the solution of the boundary value problem
\begin{equation}
\label{defn of v_tau}
            \begin{cases}
-\nabla\cdot (A(x) \nabla v_\tau) = 0  & \text{in } \Omega' ,
\\
v_\tau = P(x,y_\tau) & \text{on } \partial\Omega' .
\end{cases}
\end{equation}
Then, we fix $\varphi$ to be an arbitrary function lying in $\mathcal{C}_0^\infty ([-1,1])$ with $\|\varphi\|_{L^2((-1,1))} = 1$, and for $t_0 \in (0,T)$, we define the  cutoff function $\varphi_\tau$ by 
\begin{equation}
\label{defn of varphi_tau}
    \varphi_\tau (t) = \sqrt{a_\tau}\varphi\left( a_\tau(t-t_0) \right),
\end{equation}
where we take $a_\tau = |\ln\tau|^{\frac{1}{16}}$ \footnote{Such choice is not optimal but suffices for the proof.}. Note that supp$(\varphi_\tau) \subset (t_0-a_\tau^{-1},t_0 + a_\tau^{-1})$ and $\lim_{\tau\to0}a^{-1}_\tau = 0$. Hence, for any fixed $t_0 \in (0,T)$,
we can find $\tau_0 \in (0, \delta)$ such that for all $\tau < \tau_0$, we have $a_\tau^{-1} < \min(t_0,T-t_0)$ and hence supp$(\varphi_\tau) \subset (0,T)$. 

\begin{lem}
\label{lemma: pseudo mollifier}
    Suppose $t_0 \in (0, T)$ and $f \in \mathcal{C}([0,T])$, and let $\varphi_\tau$ be defined as above. Then
    \begin{equation}
        \lim_{\tau \to 0} \int_0^T \varphi^2_\tau(t) f(t) dt  = f(t_0).
    \end{equation}
\end{lem}
\begin{proof}
By choosing $\tau<\tau_0$ and changing variable $u=a_\tau(t-t_0)$, we get 
$$ \int_0^T \varphi^2_\tau(t) f(t) dt  =\int_{\mathbb R}\varphi^2_\tau(t) f(t)dt= \int_{\mathbb R} a_\tau \varphi^2(a_\tau (t-t_0)) f(t) dt =   \int_{-1}^1 \varphi^2 (u) f(ua_\tau^{-1} + t_0)  du. $$
The result then follows from the dominated convergence theorem, since $\|\varphi\|_{L^2((-1,1))} = 1$. 
\end{proof}

For $t_0 \in (0,T)$ and $\tau \in (0,\tau_0)$, we define a boundary function $g_\tau$ by
\begin{equation}
    g_\tau(x,t) = \varphi_{\tau}(t)(P(x,y_\tau)-v_\tau(x)).
\end{equation}
Note that $g_\tau \in \mathcal{C}^\infty_0(0,T;H^{3/2}(\partial\Omega))\subset H^{1/2,1/2}(\Sigma)$ and supp$(g_\tau)\subset S\times(0,T)$. For $\lambda\in\R$, $t_0 \in (0,T)$ and $\tau \in (0, \tau_0)$, by \cite[Chapter 4, Theorem 6.1]{LM2}, we denote by $w_{\lambda,\tau} \in H^1(0,T;L^2(\Omega))\cap L^2(0,T;H^2(\Omega))$ the solution to problem \eqref{linearised eqn. assoc. w/ DN map} with $g = g_\tau$, i.e. 
\begin{equation}
\label{defn of w_lambda,rho,tau}
            \begin{cases}
\rho_\lambda(t) \partial_t w_{\lambda,\tau} - \gamma_\lambda(t)\nabla\cdot(A(x)\nabla w_{\lambda,\tau}) = 0   & \text{in } Q ,
\\
w_{\lambda,\tau} = g_\tau(x,t) & \text{on } \Sigma , \\
w_{\lambda,\tau}(x,0) = 0 & x \in \Omega, 
\end{cases}
\end{equation}
with $\rho_\lambda(t) = \rho(t,\lambda)$ and $\gamma_\lambda(t)=\gamma(t,\lambda)$. Now we derive some estimates for the above functions in terms of $\tau$. We first consider the case $n \geq 3$ and treat the case $n=2$ separately. 

\begin{prop}
\label{prop:estimates for various functions}
    Let $n\geq 3$. Then, for any $t_0 \in (0,T)$ and any $\tau \in (0,\tau_0)$ the functions $H(\cdot,y_\tau)$ and $\mathcal{R}(\cdot,y_\tau)$ satisfy the estimates
\begin{equation}
\label{h and r estimates for n geq 3}
\begin{aligned}
        &\|H(\cdot,y_\tau)\|_{L^2(\Omega)} \leq \tilde{C} \tau^{2-\frac{n}{2}} ,\, \|\nabla H (\cdot,y_\tau)\|_{L^2(\Omega)} \leq \tilde{C} \tau^{1-\frac{n}{2}} , \\
        &\|\mathcal{R}(\cdot,y_\tau)\|_{L^2(\Omega)} \leq \tilde{C} \tau^{\frac{5}{2}-\frac{n}{2}} , \, \|\nabla \mathcal{R} (\cdot,y_\tau)\|_{L^2(\Omega)} \leq \tilde{C} \tau^{\frac{3}{2}-\frac{n}{2}} .
\end{aligned}
\end{equation}
Further, the functions $v_\tau$ and $g_\tau$ satisfy the estimates
    \begin{equation}
            \|v_\tau\|_{H^1(\Omega')} \leq \tilde{C}, \, \|g_\tau\|_{H_0^{1/2,1/2}(S\times (0,T))} \leq \tilde{C}  \tau^{1-\frac{n}{2}}.
    \end{equation}
Here the constant $\tilde{C}=\tilde C(\Omega,S,T,A)>0$. 
\end{prop}
\begin{proof}
To bound  $H(\cdot,y_\tau)$, by taking sufficiently large $R>0$ such that $\Omega \subset B(y_\tau,R)$ and noting that $\Omega \subset B(y_\tau,R)\setminus B(y_\tau,\tau)$ since dist$(y_\tau,\partial\Omega) = \tau$, and using  \eqref{reference estimate for H}, we obtain
$$
\label{estimate for H} 
\|H(\cdot,y_\tau)\|^2_{L^2(\Omega)} \leq \tilde{C}\int_{B(y_\tau,R)\setminus B(y_\tau,\tau)} |x-y_\tau|^{4-2n}  dx \leq \tilde{C} \int_\tau^R r^{4-2n} r^{n-1} dr .
$$
Hence, $\|H(\cdot,y_\tau)\|_{L^2(\Omega)} \leq \tilde{C} \tau^{2-\frac{n}{2}}$. The other estimates in \eqref{h and r estimates for n geq 3} follow similarly from \eqref{reference estimate for H} and \eqref{reference estimate for R}. For $v_\tau$, the definition \eqref{defn of v_tau} and a typical estimate  \cite[Section 8.2]{gilbarg1977elliptic} gives
$$
\label{estimate for v_tau}
   \|v_\tau\|_{H^1(\Omega')} \leq \tilde{C} \|P(\cdot,y_\tau)\|_{H^{1/2}(\partial\Omega')} \leq \tilde{C} ,
$$
since $\mathrm{dist}(y_\tau,\partial\Omega')>0$ and $P(\cdot,y_\tau) \in \mathcal{C}^2(\Omega_* \setminus \{y_\tau\})$. Lastly, for $g_\tau$, we first derive 
 \begin{equation}
 \label{estimates for varphi}
     \|\varphi_\tau\|_{L^2((0,T))} \leq \tilde{C} , \, \|\partial_t\varphi_\tau\|_{L^2((0,T))} \leq  \tilde{C} a_\tau .
 \end{equation}
Last, by applying Lemma \ref{lem:lifting}, the estimates \eqref{h and r estimates for n geq 3}, and the estimate for $v_\tau$, we have 
 \begin{align*}
     & \quad  \|g_\tau\|_{H_0^{1/2,1/2}(S\times (0,T))} \leq \tilde{C}\|\varphi_{\tau}(P(\cdot,y_\tau)-v_\tau) \|_{H^1(Q)} \\
        &\leq \tilde{C}(\|\varphi_\tau\nabla(P(\cdot,y_\tau)-v_\tau)\|_{L^2(Q)} + \|\partial_t\varphi_\tau(P(\cdot,y_\tau)-v_\tau)\|_{L^2(Q)}) \\
        & \leq \tilde{C}\left(\max(1,\tau^{1-\frac{n}{2}}) + a_\tau\max(1,\tau^{2-\frac{n}{2}})\right) \leq \tilde{C} \tau^{1-\frac{n}{2}}.
        \end{align*}
\end{proof}
\begin{prop}
\label{prop:h est.}
    Let $R>0$, and consider $\lambda \in [-R,R]$,  $t_0 \in (0,T)$, and $\tau \in (0,\tau_0)$. We can decompose $w_{\lambda,\tau}$ by
    \begin{equation}\label{dede}
        w_{\lambda,\tau}(x,t) = \varphi_\tau (t) (P(x,y_\tau) - v_\tau(x)) + h_{\lambda,\tau}(x,t) ,\quad (x,t)\in Q,
    \end{equation}
    with $h_{\lambda,\tau} \in H^1(0,T;L^2(\Omega))\cap L^2(0,T;H^2(\Omega)\cap H^1_0(\Omega))$. Moreover, for $n\geq 3$, we have 
\begin{equation}
\label{eqn: est for h,dtw and B}
\begin{aligned}
     &\| h_{\lambda,\tau}\|_{H^1(Q)} \leq Ca_\tau(1+\tau^{2-\frac{n}{2}}), \, \|\partial_t w_{\lambda,\tau}\|_{L^2(Q)} \leq Ca_\tau(1+\tau^{2-\frac{n}{2}}), \\
&\|\nabla(\varphi_\tau (\mathcal{R}(\cdot,y_\tau)-v_\tau) + h_{\lambda,\tau})) \|_{L^2(Q)} \leq C( a_\tau \tau^{2-\frac{n}{2}} + \tau^{\frac{3}{2}-\frac{n}{2}} ) ,
    \end{aligned}
\end{equation}
with the constant $C=C(\Omega,\gamma,S,T,A,R,\rho)>0$.
\end{prop}
\begin{proof}
Note that $h_{\lambda,\tau}$ solves the IBVP
$$
       \begin{cases}
           \rho_\lambda(t) \partial_th_{\lambda,\tau} - \gamma_\lambda(t)\nabla \cdot (A(x) \nabla h_{\lambda,\tau}) = - \rho_\lambda(t)   \partial_t\varphi_\tau(t)  (P(x,y_\tau) - v_\tau(x))  & \text{in } Q, \\
           h_{\lambda,\tau} = 0 & \text{on } \Sigma, \\
          h_{\lambda,\tau}(x,0) = 0
           & x \in \Omega, 
       \end{cases}
$$
where the first equation follows since $\nabla \cdot (A(x)\nabla v_\tau)$ and $\nabla_x \cdot (A(x) \nabla_x P(x,y_\tau))$ both vanish on $Q$; and  the third equation follows from the assumption $\mathrm{supp}(\varphi_\tau) \subset (0,T)$. In view of \cite[Chapter 4, Theorem 6.1]{LM2}, $ \rho_\lambda \partial_t \varphi_\tau (P(\cdot,y_\tau)-v_\tau)\in L^2(Q)$, which implies 
$h_{\lambda,\tau} \in H^1(0,T;L^2(\Omega))\cap L^2(0,T;H^2(\Omega))$.
In addition, by \cite[Section 7, Theorem 5]{evans2022partial}, we get
$$
    \| h_{\lambda,\tau}\|_{H^1(Q)} \leq C\| \rho_\lambda \partial_t  \varphi_\tau  (P(\cdot,y_\tau) - v_\tau)\|_{L^2(Q)} \leq Ca_\tau\|P(\cdot,y_\tau)-v_\tau\|_{L^2(\Omega)} ,
$$
where the second inequality follows from the regularity of $\rho$ and  \eqref{estimates for varphi}. Last, by Proposition \ref{prop:estimates for various functions}, we obtain the  estimate for $h_{\lambda,\tau}$. Next, to bound $\partial_tw_{\lambda,\tau}$, we derive 
$$
\begin{aligned}
    &\quad\|\partial_tw_{\lambda,\tau}\|_{L^2(Q)}  = \|\partial_t\varphi_\tau (P(\cdot,y_\tau) - v_\tau) + \partial_th_{\lambda,\tau}\|_{L^2(Q)}
    \\ 
    & \leq Ca_\tau(\|H(\cdot,y_\tau)\|_{L^2(Q)} + \|\mathcal{R}(\cdot,y_\tau)\|_{L^2(Q)} + 
    \|v_\tau\|_{L^2(Q)} )+ 
    \|h_{\lambda,\tau}\|_{H^1(Q)} \\
    & \leq Ca_\tau(\tau^{2-\frac{n}{2}} + \tau^{\frac{5}{2}-\frac{n}{2}} + 1 + 1+\tau^{2-\frac{n}{2}}) \leq Ca_\tau(1+\tau^{2-\frac{n}{2}}) .
\end{aligned}
$$
The last estimate is derived in the same way and the proof is thus omitted. 
\end{proof}

We also consider $\overline{w}_{\lambda,\tau}$ the solution to the adjoint problem
\begin{equation}
\label{defn of w*_lambda,rho,tau}
            \begin{cases}
-\partial_t(\rho_\lambda(t)  \overline{w}_{\lambda,\tau} )- \gamma_\lambda(t)\nabla\cdot(A(x)\nabla \overline{w}_{\lambda,\tau}) = 0   & \text{in } Q ,
\\
\overline{w}_{\lambda,\tau} = g_\tau(x,t) & \text{on } \Sigma , \\
\overline{w}_{\lambda,\tau}(x,T) = 0 & x \in \Omega.
\end{cases}
\end{equation}
The next proposition asserts that $\overline{w}_{\lambda,\tau}$ can be split similarly to $w_{\lambda,\tau}$, and moreover satisfies similar estimates to Proposition \ref{prop:h est.}. The proof is analogous and thus is omitted. 
\begin{prop}
\label{prop:hbar est.}
    Let $R>0$, and consider $\lambda \in [-R,R]$, $t_0 \in (0,T)$, and $\tau \in (0, \tau_0)$. We may decompose $\overline{w}_{\lambda,\tau}$ into 
    \begin{equation}
        \overline{w}_{\lambda,\tau}(x,t) = \varphi_\tau (t) (P(x,y_\tau) - v_\tau(x)) + \overline{h}_{\lambda,\tau}(x,t),\quad (x,t)\in Q.
    \end{equation}
     with $\overline{h}_{\lambda,\tau}\in H^1(0,T;L^2(\Omega))\cap L^2(0,T;H^2(\Omega)\cap H^1_0(\Omega))$.
Moreover, for $n\geq 3$, it holds that
\begin{equation}
\begin{aligned}
\label{eqn: est for h adj, w adj and B adj}
     &\| \overline{h}_{\lambda,\tau}\|_{H^1(Q)} \leq Ca_\tau(1+\tau^{2-\frac{n}{2}}),\,\|\overline{w}_{\lambda,\tau}\|_{L^2(Q)} \leq Ca_\tau(1+\tau^{2-\frac{n}{2}}), \\
   & \|\nabla(\varphi_\tau (\mathcal{R}(\cdot,y_\tau) -v_\tau) + \overline{h}_{\lambda,\tau}) \|_{L^2(Q)} \leq C( a_\tau \tau^{2-\frac{n}{2}} + \tau^{\frac{3}{2}-\frac{n}{2}} ) ,
    \end{aligned}
\end{equation}
with $C=C(\Omega,S,T,A,R,\gamma,\rho)>0$.
\end{prop}
Now, consider the case $n=2$. The main difference from  $n \geq 3$  is that the fundamental solution $H$ has a logarithmic bound. Nevertheless, the proof is similar and hence omitted.
\begin{prop}
\label{prop:case of n=2}
    Let $n=2$ and let $R>0$, and consider $\lambda\in [-R,R]$, $t_0 \in (0,T)$, and $\tau\in (0,\tau_0)$. Then, for the functions $H$ and $\mathcal{R}$, we have 
   \begin{equation}
   \begin{aligned}
        &\|H(\cdot,y_\tau)\|_{L^2(\Omega)} \leq \tilde{C} \tau |\ln \tau|,\,  \|\nabla H (\cdot,y_\tau)\|_{L^2(\Omega)} \leq \tilde{C} |\ln\tau|^{\frac{1}{2}} ,\\
        &\|\mathcal{R}(\cdot,y_\tau)\|_{L^2(\Omega)} \leq \tilde{C} \tau^{\frac{3}{2}} , \, \|\nabla \mathcal{R} (\cdot,y_\tau)\|_{L^2(\Omega)} \leq \tilde{C} \tau^{\frac{1}{2}} .
        \end{aligned}
    \end{equation}
For the functions $v_\tau$ and $g_\tau$, we have
    \begin{equation}
            \|v_\tau\|_{H^1(\Omega')} \leq \tilde{C} , \,  \|g_\tau\|_{H_0^{1/2,1/2}(S\times (0,T))} \leq \tilde{C}  |\ln\tau|^{\frac{1}{2}}.
    \end{equation}
Here the constant $\tilde{C}=\tilde C(\Omega,S, T,A)>0$. Moreover, by Proposition \ref{prop:h est.}, we may decompose $w_{\lambda,\tau} = \varphi_\tau(P(\cdot,y_\tau) - v_\tau) + h_{\lambda,\tau}$ (resp. by Proposition \ref{prop:hbar est.} $\overline{w}_{\lambda,\tau}=\varphi_\tau(P(\cdot,y_\tau)-v_\tau)+\overline{h}_{\lambda,\tau}$), and there exists $C=C(\Omega,\gamma,S,T,A,R,\rho)>0$ such that
\begin{equation}
\begin{aligned}
    & \| h_{\lambda,\tau}\|_{H^1(Q)}+  \|\partial_t w_{\lambda,\tau}\|_{L^2(Q)} +  \|\nabla(\varphi_\tau (\mathcal{R}(\cdot,y_\tau)-v_\tau) + h_{\lambda,\tau})) \|_{L^2(Q)}  \leq Ca_\tau, \\
    &\| \overline{h}_{\lambda,\tau}\|_{H^1(Q)}+  \|\overline{w}_{\lambda,\tau}\|_{L^2(Q)} +  \|\nabla(\varphi_\tau (\mathcal{R}(\cdot,y_\tau)-v_\tau) + \overline{h}_{\lambda,\tau})) \|_{L^2(Q)}  \leq Ca_\tau. 
\end{aligned}
\end{equation}
\end{prop}

\section{Proof of the first stability estimate}
\label{sec: proof of gamma}
 
In this section, we prove the Lipschitz stability estimate \eqref{est1} of $\gamma$ stated in Theorem \ref{thm:est for gamma}. For this purpose, we will use the singular solutions constructed in Section \ref{section: construction of special solutions}. 

\begin{proof}[Proof of Theorem \ref{thm:est for gamma}]   Throughout this proof we denote by $P$, $H$ and $\mathcal{R}$ the maps $x\mapsto P(x,y_\tau)$, $x\mapsto H(x,y_\tau)$ and $x\mapsto \mathcal{R}(x,y_\tau)$ and, we write  $\gamma_\lambda = \gamma^1_{\lambda} - \gamma^2_{\lambda}$, $\rho_\lambda = \rho^1_{\lambda} - \rho^2_{\lambda}$. We will also denote $\Lambda_\lambda = \Lambda_{\lambda,\gamma^1,\rho^1}- \Lambda_{\lambda,\gamma^2,\rho^2}$. In view of Proposition \ref{prop: derivative of N is Lambda}, for every  $g \in J_S$, we have $\partial_s \mathcal{N_\lambda}(sg)\mid_{s=0} = \Lambda_{\lambda}(g)$, and by definition $\mathcal{N}_\lambda(0) = 0$. Therefore, we obtain
\begin{align*} 
&\quad \sup_{g\in B_1} \limsup_{k \to +\infty} \left(k\left\|\mathcal{N}_\lambda\left(\frac{g}{k}\right)\right\|_{H^{-1/2,-1/2}(S\times(0,T))}\right) \\
&=\sup_{g\in B_1}  \left\|\lim_{k \to +\infty}\left(\frac{\mathcal{N}_\lambda\left(\frac{g}{k}\right)-\mathcal{N}_\lambda(0)}{\frac{1}{k}}\right)\right\|_{H^{-1/2,-1/2}(S\times(0,T))}
= \sup_{g\in B_1} \|\Lambda_{\lambda}(g) \|_{H^{-1/2,-1/2}(S\times (0,T))}.
\end{align*}
Combining this with the density of $J_S$ in $ H^{1/2,1/2}_0(S\times(0,T))$ and fixing
$\mathbb B_1 = \{g \in  H^{1/2,1/2}_0(S\times(0,T)):\  \|g\|_{H^{1/2,1/2}_0(S\times(0,T))} \leq 1\}$ gives
$$
     \sup_{g\in B_1} \limsup_{k \to +\infty} \left(k\left\|\mathcal{N}_\lambda\left(\frac{g}{k}\right)\right\|_{H^{-1/2,-1/2}(S\times(0,T))}\right) =\sup_{g\in \mathbb B_1} \|\Lambda_{\lambda}(g) \|_{H^{-1/2,-1/2}(S\times (0,T))}= \|\Lambda_\lambda\|_{op} ,
$$
		where $\|\cdot\|_{op}$ denotes the operator norm from $H_0^{1/2,1/2}(S\times(0,T))$ to $H^{-1/2,-1/2}(S\times(0,T))$. By Proposition \ref{prop:extension of DN map and continuity},  $\sup_{\lambda \in [-R,R]} \|\Lambda_\lambda\|_{op} < \infty$. Thus by the regularity of $\gamma$, it remains to show 
    \begin{equation}\label{estlip}
        |\gamma(t_0,\lambda)| \leq \tilde{C} \sup_{\mu \in [-R,R]} \|\Lambda_\mu \|_{op},\quad (t_0,\lambda)\in[0,T]\times[-R,R], 
    \end{equation}
    with $\tilde{C}=\tilde C(\Omega,S,T,A)>0$.
Fix $\lambda\in[-R,R]$,  $t_0 \in (0,T)$, and $\tau \in (0,\tau_0)$. Let $w^j_{\lambda,\tau}$ (resp. $\overline{w}^{j}_{\lambda,\tau}$) be the  solution of \eqref{defn of w_lambda,rho,tau} (resp. \eqref{defn of w*_lambda,rho,tau}) with  $(\gamma,\rho)=(\gamma^j,\rho^j)$. By Proposition \ref{prop:h est.} (resp. Proposition \ref{prop:hbar est.}), $w^j_{\lambda,\tau}$ (resp. $\overline{w}^{j}_{\lambda,\tau}$) takes the form \eqref{dede} (resp. \eqref{defn of w*_lambda,rho,tau}). 
Fix $w_{\lambda,\tau}=w_{\lambda,\tau}^2-w_{\lambda,\tau}^1$ and note that $w_{\lambda,\tau}\in H^1(0,T;L^2(\Omega))\cap L^2(0,T;H^2(\Omega))$ solves the IBVP
$$
        \begin{cases}
\rho^2_\lambda(t) \partial_t w_{\lambda,\tau} - \gamma^2_\lambda(t)\nabla \cdot (A(x) \nabla w_{\lambda,\tau}) = \rho_\lambda(t) \partial_t w_{\lambda,\tau}^1 - \gamma_\lambda(t)\nabla \cdot (A(x)  \nabla w_{\lambda,\tau}^1)  & \text{in } Q ,
\\
w_{\lambda,\tau} = 0 & \text{on } \Sigma , \\
w_{\lambda,\tau}(x,0) = 0 & x \in \Omega.
\end{cases}
$$
    Multiplying the identity by $\overline{w}^2_{\lambda,\tau}$, integrating over $Q$ and integrating by parts give
$$
    \begin{aligned}
         \int_\Sigma - \gamma^2_\lambda(t)(A(x)  \nabla w_{\lambda,\tau} \cdot \nu(x))  \overline{w}^2_{\lambda,\tau} d\sigma(x) dt &=
      \int_Q (\rho_\lambda(t) \partial_t w_{\lambda,\tau}^1\overline{w}^2_{\lambda,\tau}  + \gamma_\lambda(t)A(x) \nabla w_{\lambda,\tau}^1 \cdot \nabla \overline{w}_{\lambda,\tau}^2)dx dt \\ & \quad \quad - \int_\Sigma\gamma_\lambda(t) (A(x)\nabla w_{\lambda,\tau}^1 \cdot \nu(x)) \overline{w}_{\lambda,\tau}^2 d\sigma(x) dt . 
    \end{aligned}
$$
    Next, we group the terms that involve the boundary integrals and note the identity 
    $$A\gamma_\lambda \nabla w_{\lambda,\tau}^1 \cdot \nu - A \gamma^2_\lambda \nabla w_{\lambda,\tau} \cdot \nu = A\gamma^1_\lambda \nabla w_{\lambda,\tau}^1 \cdot \nu - A \gamma^2_\lambda \nabla w_{\lambda,\tau}^2 \cdot \nu.$$
    Moreover, we have $w_{\lambda,\tau}^j = g_\tau = \overline{w}_{\lambda,\tau}^j$  on $\Sigma$ and $\gamma^j_\lambda (A\nabla w^j_{\lambda,\tau} \cdot \nu) = \Lambda_{\lambda,\gamma^j,\rho^j}g_\tau $ (we may use such definition of $\Lambda_{\lambda,\gamma^j,\rho^j}$ instead of \eqref{defn of DN map for less smooth functions}, since $w_{\lambda,\tau}\in L^2(0,T;H^2(\Omega))$).
    Hence, we derive
$$
\label{main identity}
        \int_\Sigma (\Lambda g_\tau)g_\tau d\sigma(x) dt = \int_Q (\rho_\lambda(t) \partial_t w_{\lambda,\tau}^1\overline{w}^2_{\lambda,\tau}  + \gamma_\lambda(t)A(x) \nabla w_{\lambda,\tau}^1 \cdot \nabla \overline{w}_{\lambda,\tau}^2)dx dt  .
$$
By Propositions \ref{prop:h est.} and \ref{prop:hbar est.}, and recalling $P=H+\mathcal{R}$, we may rewrite the identity as  
\begin{align*}
           \int_\Sigma (\Lambda g_\tau)g_\tau d\sigma(x) dt &= \int_Q \rho_\lambda(t) \partial_t w_{\lambda,\tau}^1\overline{w}^2_{\lambda,\tau}\ dx dt \\ &+ \int_Q \gamma_\lambda(t) A(x)\nabla(\varphi_\tau(t) H+B) \cdot \nabla(\varphi_\tau(t)H+\overline{B}) dx dt  ,
\end{align*}
where $B = \varphi_\tau(t)(\mathcal{R}- v_\tau(x)) + h^1_{\lambda,\tau}$ (resp. $\overline{B} = \varphi_\tau(t) (\mathcal{R}- v_\tau(x)) + \overline{h}^2_{\lambda,\tau} $). We rearrange the identity to derive
\begin{align*}
    &\int_Q \gamma_\lambda(t) \varphi_\tau^2(t) A(x)\nabla H \cdot \nabla H dx dt \\
    =  &\int_\Sigma (\Lambda_\lambda g_\tau)g_\tau d\sigma(x) dt -  \int_Q \rho_\lambda(t) \partial_t w_{\lambda,\tau}^1\overline{w}^2_{\lambda,\tau} dx dt 
    - \int_Q \gamma_\lambda(t) \varphi_\tau(t)A(x)\nabla B \cdot \nabla H dxdt\\
    &- \int_Q \gamma_\lambda(t)\varphi_\tau(t)A(x) \nabla  H \cdot \nabla \overline{B} dxdt 
    - \int_Q \gamma_\lambda(t) A(x) \nabla B \cdot \nabla \overline{B}dx  dt .
\end{align*}
First consider the case $n \geq 3$, for which we bound the terms using Propositions \ref{prop:estimates for various functions} - \ref{prop:hbar est.}. There exists $\tilde{C}=\tilde{C}(\Omega,S,T, A)>0$ and $C=C(\Omega,\gamma,S,T,A,R,\rho)>0$ such that 
\begin{align*}
    \left\lvert   \int_\Sigma (\Lambda_\lambda g_\tau)g_\tau d\sigma(x) dt\right\lvert &=\left\lvert\left\langle \Lambda_\lambda g_\tau, g_\tau\right\rangle_{H^{-1/2,-1/2}(S\times(0,T)),H_0^{1/2,1/2}(S\times(0,T))}\right\lvert\\
		&\leq \|\Lambda_\lambda g_\tau\|_{H^{-1/2,-1/2}(S\times(0,T))} \|g_\tau\|_{H^{1/2,1/2}_0(S\times (0,T))}\\ & \leq \|\Lambda_\lambda\|_{op} \|g_\tau\|^2_{H^{1/2,1/2}(S\times (0,T))} \leq \tilde{C}\|\Lambda_\lambda\|_{op} \tau^{2-n} , \\
     \left\lvert     \int_Q \rho_\lambda(t) \partial_t w_{\lambda,\tau}^1\overline{w}^2_{\lambda,\tau}  dx dt\right\lvert  &\leq \|\rho_\lambda\|_{L^\infty((0,T))}\|\partial_t w_{\lambda,\tau}^1\|_{L^2(Q)} \|\overline{w}_{\lambda,\tau}^2\|_{L^2(Q)} \leq C a^2_\tau (1+\tau^{4-n}) . 
\end{align*}
Additionally, we see that the terms involving $\nabla B$ and $\nabla \overline{B}$ can be estimated using \eqref{eqn: est for h,dtw and B} and \eqref{eqn: est for h adj, w adj and B adj} respectively. Utilizing Fubini's theorem, we derive
\begin{align*}
   \left\lvert \int_Q  \gamma_\lambda(t)\varphi_\tau(t) A(x)\nabla B \cdot \nabla H  dx dt\right\lvert &\leq C\|\gamma_\lambda\|_{L^\infty((0,T))} \|A\|_{L^\infty(\Omega)} \|\nabla B\|_{L^2(Q)} \|\varphi_\tau\nabla H\|_{L^2(Q)} \\
   &\leq C(a_\tau \tau^{2-\frac{n}{2}}+\tau^{\frac{3}{2}-\frac{n}{2}} )\|\varphi_\tau\|_{L^2((0,T))} \|\nabla H\|_{L^2(\Omega)} 
    \leq C\tau^{\frac{5}{2}-n} ,\\
   \left\lvert \int_Q  \gamma_\lambda(t) \varphi_\tau(t) A(x)\nabla H \cdot \nabla \overline{B} dx dt \right\lvert &\leq C\|\gamma_\lambda\|_{L^\infty((0,T))} \|A\|_{L^\infty(\Omega)} \|\varphi_\tau\nabla H\|_{L^2(Q)} \|\nabla \overline{B}\|_{L^2(Q)} \\
   &\leq C(a_\tau \tau^{2-\frac{n}{2}}+\tau^{\frac{3}{2}-\frac{n}{2}} )\|\varphi_\tau\|_{L^2((0,T))} \|\nabla H\|_{L^2(\Omega)} 
   \leq C\tau^{\frac{5}{2}-n} ,\\
   \left\lvert \int_Q  \gamma_\lambda(t) A(x)\nabla B \cdot \nabla \overline{B}  dx dt\right\lvert &\leq C\|\gamma_\lambda\|_{L^\infty((0,T))} \|A\|_{L^\infty(\Omega)} \|\nabla B\|_{L^2(Q)} \|\nabla \overline{B}\|_{L^2(Q)} \\ &\leq C(a_\tau^2 \tau^{4-n} + \tau^{3-n}) .
   \end{align*}
Combining the preceding estimates and again applying Fubini's theorem yield
\begin{equation}
 \label{aux_4,1}
\begin{aligned}
  &\quad \left \lvert \int_0^T \gamma_\lambda(t) \varphi_\tau^2(t)dt \right \lvert \left\lvert \int_{\Omega} A(x) \nabla H \cdot \nabla H dx\right \lvert\\
  &=\left \lvert \int_Q \gamma_\lambda(t) \varphi_\tau^2(t) A(x) \nabla H \cdot \nabla H dx dt\right \lvert \\
&\leq \tilde{C}\|\Lambda_\lambda\|_{op} \tau^{2-n} + C( a_\tau \tau^{4-n} + \tau^{\frac{5}{2}-n} + a_\tau^2 \tau^{4-n} + \tau^{3-n}) \\ &
\leq \tilde{C}\|\Lambda_\lambda\|_{op}\tau^{2-n} +C \tau^{\frac{5}{2}-n} . 
\end{aligned}
\end{equation}
Last, by \eqref{reference estimate for H} and utilizing polar coordinates, we get
\begin{equation}
\label{eqn:lower bound for gradH}
\begin{aligned}
    &\quad\left\lvert \int_\Omega A(x) \nabla H \cdot \nabla Hdx \right\lvert\geq c\int_\Omega |\nabla H|^2dx \geq \tilde{C} \int_\Omega |x-y_\tau|^{2-2n} dx  \\& 
    \geq \tilde{C} \int_{[B(y_\tau,2\tau) \setminus B(y_\tau,\tau) ]\cap \Omega}  |x-y_\tau|^{2-2n} \geq \tilde{C} \int_\tau^{2\tau} r^{1-n} \ dr  \geq \tilde{C} \tau^{2-n}.
    \end{aligned}
\end{equation}
Hence, combining the estimates \eqref{aux_4,1} and \eqref{eqn:lower bound for gradH} gives 
$$
     \left \lvert \int_0^T \gamma_\lambda(t) \varphi_\tau^2(t) dt\right \lvert \leq \tilde{C} \|\Lambda_\lambda\|_{op} + C\tau^{\frac{1}{2}} .
$$
Taking $\tau\to 0$ and applying Lemma \ref{lemma: pseudo mollifier} leads to
\begin{equation}\label{estint}
    |\gamma(t_0,\lambda)| \leq \tilde{C} \|\Lambda_\lambda\|_{op} \leq \tilde{C} \sup_{\mu \in [-R,R]}\|\Lambda_{\mu}\|_{op}.
\end{equation} 
Since $\tilde{C}$ is independent of  $\lambda\in[-R,R]$ and $t_0\in(0,T)$, by the continuity of $\gamma$, the estimate \eqref{estint}  holds for any $(t_0,\lambda) \in [0,T]\times[-R,R]$. This proves \eqref{estlip} and it completes the proof for the case $n \geq 3$. The case $n=2$ can be treated similarly by applying Proposition \ref{prop:case of n=2}, and we only sketch the main steps. By repeating the previous argumentation, we obtain
$$
\begin{aligned}
    \left \lvert \int_0^T \gamma_\lambda(t) \varphi_\tau^2(t) dt \right \lvert \left\lvert \int_{\Omega} A(x) \nabla H \cdot \nabla H dx\right \lvert &\leq \tilde{C} \|\Lambda_\lambda\|_{op}|\ln \tau| + C(a_\tau^2 + a_\tau |\ln \tau|^{\frac{1}{2}} + a_\tau^2)\\ &\leq \tilde{C} \|\Lambda \|_{op} |\ln \tau| + C a_\tau |\ln \tau|^{\frac{1}{2}} .
\end{aligned}
$$
Then similarly, we may derive
$\left\lvert \int_\Omega A(x) \nabla H \cdot \nabla Hdx \right\lvert  \geq \tilde{C} |\ln \tau| .$
Consequently, we have 
$$
      \left \lvert \int_0^T \gamma_\lambda (t)\varphi_\tau^2(t) dt\right \lvert \leq \tilde{C} \|\Lambda_\lambda\|_{op} + Ca_\tau|\ln \tau|^{-\frac{1}{2}}.
$$
We obtain the estimate for $n=2$ by taking $\tau \to 0$ and repeating the previous argument.
\end{proof}

\section{Proof of the second stability estimate}
\label{sec: proof of rho}

This section is devoted to the  proof of the H\"older stability estimate \eqref{estt} of $\rho$ stated in Theorem \ref{thm:est of rho}. For this purpose, we will use a new class of  singular solutions of the linearized equation \eqref{linearised eqn. assoc. w/ DN map}. These solutions can be compared with those introduced in Section \ref{section: construction of special solutions} under some modifications related to the new class of parameter $\rho$ that we want to determine and the corresponding H\"older stability estimate \eqref{estt}. In all this section, we assume that $A = \text{Id}$ and  $n \geq 3$. We also additionally assume that 
\begin{equation}\label{cond}
    \sup_{t\in [0,T]}\partial_t \rho(t,s) \leq \kappa(s), \quad s\in\R,
\end{equation}
for some strictly positive continuous function $\kappa:\R\to(0,\infty)$.
We now adjust the singular solutions constructed in Section \ref{section: construction of special solutions}. In this case, for all $(x,y)\in \Omega_*\times\Omega_*\setminus D$, we can fix  $P(x,y)=H(x,y) = |x-y|^{2-n}$ and  observe that, for all $y\in\Omega_*$,  $\Delta H(\cdot,y)=0$ on $\Omega_*\setminus\{y\}$. Also take $y_\tau \in \Omega'$ for $\tau \in (0,\delta)$ as before, and by fixing $t_0\in(0,T)$, consider the cutoff function $\varphi_\tau$ defined by \eqref{defn of varphi_tau} which is modified with $a_\tau = \tau^{-r}$, where $r>0$ is to be determined. By fixing $j=1,\ldots,n$, we define $v_{j,\tau} \in H^2(\Omega')$ the solution of 
\begin{equation}
    \begin{cases}
        -\Delta v_{j,\tau} = 0 & \text{in } \Omega', \\
        v_{j,\tau} = \partial_{x_j}H(x,y_\tau) & \text{on } \partial\Omega'.
    \end{cases}
\end{equation}
Fix $\tau_0\in(0,\delta)$ such that $\tau_0^r < \min(t_0,T-t_0)$. Then, for $\tau < \tau_0$, we construct the modified boundary functions $g_{j,\tau} \in H_0^{1/2,1/2}(S\times(0,T))$ by 
\begin{equation}
    g_{j,\tau}(x,t) = \chi(t) \Phi_\tau(t) (\partial_{x_j}H(x,y_\tau)-v_{j,\tau}(x)) ,
\end{equation}
where $\Phi_\tau(t) = \int_0^t \varphi_\tau (s) ds  $, and $\chi \in \mathcal{C}_0^\infty([0,T])$ is an arbitrary function satisfying $\chi=1$ 
on  $[t_0-\tau_0^r,t_0+\tau_0^r]$. Since $\tau<\tau_0$, we have supp$(\varphi_\tau) \subset [t_0-\tau^r,t_0+\tau^r]\subset(t_0-\tau_0^r,t_0+\tau_0^r)$, which implies $\chi = 1$ on a neighborhood of $\mathrm{supp}(\varphi_\tau$). Note also that
\begin{equation}
\label{expression for chiPhi and (chiPhi)_t}
\chi(t) \Phi_\tau(t) = \Phi_\tau(t) + \xi(t) , \, \partial_t(\chi(t) \Phi_\tau(t)) = \varphi_\tau(t) + \zeta(t) ,    
\end{equation}
where $\xi = \Phi_\tau(T)\chi I_{[t_0+\tau^r,T]}$ and $\zeta = \Phi_\tau(T)(\partial_t\chi)I_{[t_0+\tau^r,T]}$, with $I_B$ denoting the indicator function of the set $B$. Due to \eqref{estimates for varphi}, there holds $\|\Phi_\tau\|_{L^2((0,T))} \leq C$ for $C>0$ independent of $\tau$. Hence, $\|\xi\|_{L^2((0,T))}+\|\zeta\|_{L^2((0,T))} \leq C$.
In addition, since $\chi = 1$ in a neighborhood of $\text{supp}(\varphi_\tau$), we have  $\xi\varphi_\tau = \zeta\varphi_\tau \equiv 0$. Finally, for $\lambda\in\R$ and $\tau \in (0,\tau_0)$, we denote by $w_{j,\lambda,\tau} \in H^1(0,T;L^2(\Omega))\cap L^2(0,T;H^2(\Omega))$ the solution to problem \eqref{linearised eqn. assoc. w/ DN map} with $g = g_{j,\tau}$, i.e. 
\begin{equation}
    \begin{cases}
        \rho_\lambda(t)\partial_t w_{j,\lambda,\tau} - \gamma_{\lambda}(t) \Delta w_{j,\lambda,\tau} = 0 & \text{in }Q, \\
        w_{j,\lambda,\tau} = g_{j,\tau}(x,t) & \text{on }\Sigma, \\
        w_{j,\lambda,\tau}(x,0) = 0 & x \in \Omega. 
    \end{cases}
\end{equation}
\begin{prop}
\label{prop: w for rho case}
Let  $R>0$ and consider $\lambda\in[-R,R]$. For any $t_0 \in (0,T)$ and any  $\tau \in (0,\tau_0)$,  $w_{j,\lambda,\tau}$ can be decomposed into
    \begin{equation}
        w_{j,\lambda,\tau}(x,t) =  \chi(t) \Phi_\tau(t) (\partial_{x_j}H(x,y_\tau) -v_{j,\tau}(x)) + h_{j,\lambda,\tau}(x,t),\quad (x,t)\in Q  ,
    \end{equation}
    with $h_{j,\lambda,\tau} \in H^1(0,T;L^2(\Omega))\cap L^2(0,T;H^2(\Omega)\cap H^1_0(\Omega))$.
    Moreover, we have the estimates
    \begin{equation}
    \begin{aligned}
        &\|g_{j,\tau}\|_{H_0^{1/2,1/2}(S \times (0,T))} \leq C \tau^{-\frac{n}{2}} , \,           \|w_{j,\lambda,\tau}\|_{H^1(Q)} \leq \tau^{-\frac{n}{2}} ,\\
             &\|h_{j,\lambda,\tau}\|_{H^1(0,T;H^{-1}(\Omega))}+ \|h_{j,\lambda,\tau}\|_{L^2(0,T;H^1(\Omega))}\leq C(1+\tau^{2-\frac{n}{2}} ),
            \end{aligned}
        \end{equation}
with the constant $C=C(\Omega, S, T, R, \kappa,m)>0$.
\end{prop}
\begin{proof} 
We first bound $g_{j,\tau}$.
By Lemma \ref{lem:lifting} and Proposition \ref{prop:estimates for various functions}, we have 
\begin{align*}
    &\quad \|g_{j,\tau}\|_{H_0^{1/2,1/2}(S \times (0,T))}  \leq C\left\lvert\left\lvert 
 \chi \Phi_\tau (\partial_{x_j}H(\cdot,y_\tau)-v_{j,\tau})\right\lvert \right\lvert_{H^1(Q)} \\
 & \leq C\left(\|\partial_t(\chi\Phi_\tau)(\partial_{x_j}H(\cdot,y_\tau)-v_{j,\tau})\|_{L^2(Q)} + \left\lvert\left\lvert \chi\Phi_\tau\partial_{x_k}(\partial_{x_j} H(\cdot,y_\tau)-v_{j,\tau})\right\lvert \right\lvert_{L^2(Q)}\right) \\
 & \leq C\left(\|(\varphi_\tau + \zeta)(\partial_{x_j}H(\cdot,y_\tau)-v_{j,\tau})\|_{L^2(Q)} + \|(\Phi_\tau + \xi)\partial_{x_k}(\partial_{x_j} H(\cdot,y_\tau)-v_{j,\tau})\|_{L^2(Q)} \right)\\
 & \leq C(\tau^{1-\frac{n}{2}} + 1+ \tau^{-\frac{n}{2}} +1 ) \leq C\tau^{-\frac{n}{2}},
    \end{align*}
where the bound $\|\partial_{x_k}\partial_{x_j}H(\cdot,y_\tau)\|_{L^2(\Omega)} \leq C\tau^{-\frac{n}{2}}$ follows similarly as Proposition \ref{prop:estimates for various functions}. Next, consider the estimate for $h_{j,\lambda,\tau}$. Noting that  $\Delta \partial_{x_j} H(x,y_\tau)=\partial_{x_j}\Delta  H(x,y_\tau) = 0 $ for $x \in \Omega$, we see $h_{j,\lambda,\tau}$ solves
$$
    \begin{cases}
        \rho_\lambda(t) \partial_t h_{j,\lambda,\tau} -\gamma_\lambda(t) \Delta h_{j,\lambda,\tau} = - \rho_\lambda(t)\partial_t(\chi(t)\Phi_\tau (t))(\partial_{x_j} H(x,y_\tau)-v_{j,\tau}(x)) & \text{in }Q, \\
        h_{j,\lambda,\tau} = 0 & \text{on } \Sigma, \\
        h_{j,\lambda,\tau}(x,0) = 0  & x\in \Omega.
    \end{cases}
$$
 Like the proof of Lemma \ref{lem:existence of solution with less regular boundary data}, by \cite[Chapter 3, Theorem 4.1]{LM1}, we obtain
\begin{align*}
&\quad\|h_{j,\lambda,\tau}\|_{H^1(0,T;H^{-1}(\Omega))}+\|h_{j,\lambda,\tau}\|_{L^2(0,T;H^1(\Omega))} \leq \| \rho_\lambda(\varphi_\tau + \zeta)(\partial_{x_j} H(\cdot,y_\tau)-v_{j,\tau})\|_{L^2(0,T;H^{-1}(\Omega))}\\
&\leq C(1+\| \partial_{x_j} H(\cdot,y_\tau)\|_{H^{-1}(\Omega)})	\leq C(1+\|  H(\cdot,y_\tau)\|_{L^2(\Omega)})\leq C(1+\tau^{2-\frac{n}{2}} ) .\end{align*}
Lastly, combining the above estimates gives the bound on $w_{j,\lambda,\tau}$. 
\end{proof}

Also we define the boundary data
\begin{equation}
    \overline{g}_{j,\tau}(x,t) = \varphi_\tau(t) (\partial_{x_j}H(x,y_\tau)-v_{j,\tau}(x)) ,
\end{equation}
and let $\overline{w}_{j,\lambda,\tau}$ be the solution to the adjoint problem 
\begin{equation}
            \begin{cases}
-\partial_t(\rho_\lambda(t)  \overline{w}_{j,\lambda,\tau} )- \gamma_\lambda(t)\Delta \overline{w}_{j,\lambda,\tau} = 0   & \text{in } Q ,
\\
\overline{w}_{j,\lambda,\tau} = \overline{g}_{j,\tau}(x,t) & \text{on } \Sigma , \\
\overline{w}_{j,\lambda,\tau}(x,T) = 0 & x \in \Omega.
\end{cases}
\end{equation}

\begin{prop}
\label{prop: w adj for rho case}
    Let  $R>0$ and consider $\lambda\in[-R,R]$. For any $t_0 \in (0,T)$ and any  $\tau \in (0,\tau_0)$, we may rewrite $\overline{w}_{j,\lambda,\tau}$ as
    \begin{equation}
        \overline{w}_{j,\lambda,\tau}(x,t) =  \varphi_\tau(t) (\partial_{x_j}H(x,y_\tau)-v_{j,\tau}(x)) + \overline{h}_{j,\lambda,\tau}(x,t),\quad (x,t)\in Q .
    \end{equation}
   with   $\overline{h}_{j,\lambda,\tau} \in H^1(0,T;L^2(\Omega))\cap L^2(0,T;H^2(\Omega)\cap H^1_0(\Omega))$. Moreover, we have
    \begin{equation}
    \begin{aligned}
        &\|\overline{g}_{j,\tau}\|_{H_0^{1/2,1/2}(S \times (0,T))} \leq C (\tau^{1-r-\frac{n}{2}}+\tau^{-\frac{n}{2}}) ,\,       \|\overline{h}_{j,\lambda,\tau}\|_{L^2(0,T;H^1(\Omega))}\leq C(\tau^{-r}+\tau^{2-r-\frac{n}{2}} ) ,\\
&\|\overline{w}_{j,\lambda,\tau}\|_{L^2(0,T;H^1(\Omega))} \leq C( \tau^{-\frac{n}{2}}+\tau^{-r}+\tau^{2-r-\frac{n}{2}}) ,
        \end{aligned}
    \end{equation}
    with the constant $C=C(\Omega, S, T, R, \kappa,m)>0$. 
\end{prop}
\begin{proof}
First we bound $\overline{g}_{j,\tau}$. Lemma \ref{lem:lifting}, Proposition \ref{prop:estimates for various functions}, and the estimate \eqref{estimates for varphi} imply 
$$
    \begin{aligned}
        &\quad \|\overline{g}_{j,\tau}\|_{H_0^{1/2,1/2}(S \times (0,T))} 
         \leq C\left\lvert\left\lvert 
 \varphi_\tau(\partial_{x_j}H(\cdot,y_\tau) - v_{j,\tau}\right\lvert \right\lvert_{H^1(Q)} \\
 & \leq C\left(\|\partial_t\varphi_\tau(\partial_{x_j}H(\cdot,y_\tau)-v_{j,\tau})\|_{L^2(Q)} + \left\lvert\left\lvert \varphi_\tau\partial_{x_k}(\partial_{x_j} H(\cdot,y_\tau)-v_{j,\tau})\right\lvert \right\lvert_{L^2(Q)} \right)\\
 & \leq C(\tau^{1-r-\frac{n}{2}} +\tau^{-r}+\tau^{-\frac{n}{2}}+1) \leq C(\tau^{1-r-\frac{n}{2}}+\tau^{-\frac{n}{2}}),
    \end{aligned}
$$
using again the estimate $\|\partial_{x_k}\partial_{x_j}H(\cdot,y_\tau)\|_{L^2(\Omega)}\leq C\tau^{-\frac{n}{2}}$. Note that $\overline{h}_{j,\lambda,\tau}$ solves
$$
    \begin{cases}
        -\partial_t(\rho_\lambda (t)\overline{h}_{j,\lambda,\tau}) - \gamma_{\lambda}(t) \Delta \overline{h}_{j,\lambda,\tau} = \partial_t(\rho_\lambda(t) \varphi_{\tau}(t))(\partial_{x_j}H(x,y_\tau) - v_{j,\tau}(x)) & \text{in }Q, \\
        \overline{h}_{j,\lambda,\tau} = 0 & \text{on } \Sigma, \\
        \overline{h}_{j,\lambda,\tau}(x,T) = 0  & x\in \Omega. 
    \end{cases}
$$
Then, by \cite[Chapter 3, Theorem 4.1]{LM1}, we have
\begin{align*}
\|\overline{h}_{j,\lambda,\tau}\|_{L^2(0,T;H^1(\Omega))} &\leq \| \partial_t(\rho_\lambda \varphi_\tau)(\partial_{x_j}H(\cdot,y_\tau)-v_{j,\tau})\|_{L^2(0,T;H^{-1}(\Omega))}\\
		&\leq C(1+\tau^{-r})(1+\| \partial_{x_j} H(\cdot,y_\tau)\|_{H^{-1}(\Omega)})\\
		&\leq C\tau^{-r}(1+\|  H(\cdot,y_\tau)\|_{L^2(\Omega)}) \leq C(\tau^{-r}+\tau^{2-r-\frac{n}{2}} ) .
\end{align*}
Now the bound on $\overline{w}_{j,\lambda,\tau}$ follows directly from the bounds for $\overline{g}_{j,\tau}$ and $\overline{h}_{j,\lambda,\tau}$. 
\end{proof}

With Proposition \ref{prop: w for rho case} and \ref{prop: w adj for rho case}, we can now prove Theorem \ref{thm:est of rho}.

\begin{proof}[Proof of Theorem \ref{thm:est of rho}] Fix $R>0$. 
Since $\rho^j \in    \mathcal{C}^1([0,T]; \mathcal{C}^3(\R))$ for $j=1,2$, we can find $(t_0,\lambda_R)\in[0,T]\times[-R,R]$ such that
\begin{equation}\label{cond3}\sup_{(t,\lambda)\in[0,T]\times[-R,R]}|\rho^1(t,\lambda)-\rho^2(t,\lambda)|=|\rho^1(t_0,\lambda_R)-\rho^2(t_0,\lambda_R)|.\end{equation}
Then, condition \eqref{cond2} implies that we may assume  $t_0\in(0,T)$. Similar to the proof of Theorem \ref{thm:est for gamma}, for $j=1,\ldots,n$,  denote by $w^k_{j,\lambda_R,\tau}$ the singular solutions in Proposition \ref{prop: w for rho case} with $(t_0,\lambda_R) \in (0,T) \times[-R,R]$ satisfying condition \eqref{cond3}, and with $\gamma=\gamma^k$, $\rho=\rho^k$, $k=1,2$. Moreover, we write $\lambda=\lambda_R$, and for $j=1,\ldots,n$, we denote $w_{j,\lambda,\tau} = w^2_{j,\lambda,\tau} - w^1_{j,\lambda,\tau}$, $\gamma_{\lambda} = \gamma^1_{\lambda} -\gamma^2_{\lambda}$, $\rho_{\lambda} = \rho^1_{\lambda} - \rho^2_{\lambda}$.  We also denote  $\Lambda_\mu = \Lambda_{\mu,\gamma^1,\rho^1} - \Lambda_{\mu,\gamma^2,\rho^2}$, $\mu\in\R$. Following the argumentation in Theorem \ref{thm:est for gamma},  we only need to show 
    \begin{equation}
        |\rho(\lambda,t_0)| \leq C\eta^{\frac{1}{9}}, \quad \mbox{with }  \eta := \sup_{\mu \in [-R,R]} \|\Lambda_\mu\|_{op} ,
    \end{equation}
    with $\|\cdot\|_{op}$ denoting the operator norm on maps from $H^{1/2,1/2}_0(S\times (0,T))$ to $H^{-1/2,-1/2}(S\times (0,T))$. Similar to Theorem \ref{thm:est for gamma}, we derive the identity
$$
        \int_\Sigma (\Lambda_{\lambda} g_{j,\tau})\overline{g}_{j,\tau} d\sigma(x) dt = \int_Q \rho_\lambda(t) \partial_t w^1_{j,\lambda,\tau} \overline{w}^2_{j,\lambda,\tau} dx dt+ \int_Q \gamma_\lambda(t) \nabla w^1_{j,\lambda,\tau} \cdot \nabla \overline{w}^2_{j,\lambda,\tau}dx dt .
$$
In view of Proposition \ref{prop: w for rho case} and \ref{prop: w adj for rho case}, and by \eqref{expression for chiPhi and (chiPhi)_t}, we can rewrite it as
\begin{equation}\label{ttt}
    \int_Q \rho_\lambda(t) \varphi_\tau^2 (t)(\partial_{x_j}H(x,y_\tau))^2 dx dt = \int_\Sigma (\Lambda_\lambda g_{j,\tau})\overline{g}_{j,\tau}d\sigma(x)dt - \int_Q \gamma_\lambda(t) \nabla w^1_{j,\lambda,\tau} \cdot \nabla \overline{w}^2_{j,\lambda,\tau}dx dt - B,
    \end{equation} 
with
\begin{align*}
B =& \int_Q\rho_\lambda(t)\varphi_\tau(t) \zeta(t)(\partial_{x_j}H(x,y_\tau)-v_{j,\tau}(x))^2 dx dt\\ 
 & + \int_Q \rho_\lambda(t) \varphi_\tau^2(t) (-2v_{j,\tau}(x) \partial_{x_j} H(x,y_\tau) + v_{j,\tau}^2(x)) dx dt \\
 &+\int_Q\rho_\lambda(t)\partial_th^1_{j,\lambda,\tau}(x,t)\varphi_\tau (t)(\partial_{x_j}H(x,y_\tau)-v_{j,\tau}(x)) dx dt 
 \\
 &+ \int_Q \rho_\lambda(t) \overline{h}^2_{j,\lambda,\tau}(x,t)(\varphi_\tau(t)+\zeta(t))(\partial_{x_j}H(x,y_\tau)-v_{j,\tau}(x)) dx dt \\
 &+ \int_Q \partial_t h^1_{j,\lambda,\tau} (x,t)  \overline{h}^2_{j,\lambda,\tau}(x,t) dx dt .
 \end{align*}
Since $\varphi_\tau\zeta\equiv0$, $\varphi_\tau\in \mathcal{C}^\infty_0([0,T])$, $\overline{h}^2_{j,\lambda,\tau}\in L^2(0,T;H^1_0(\Omega))$, integration by parts in $t$ gives
\begin{align*}B =& \int_Q\rho_\lambda(t)\varphi^2_\tau(t)( -2v_{j,\tau}(x)\partial_{x_j}H(x,y_\tau)+v^2_{j,\tau}(x))dx dt\\ 
 &-\int_Qh^1_{j,\lambda,\tau}(x,t)(\rho_\lambda(t)\partial_t\varphi_\tau(t)+\partial_t\rho_\lambda(t)\varphi_\tau (t))(\partial_{x_j}H(x,y_\tau)-v_{j,\tau}(x))dx dt\\
&+ \int_Q \rho_\lambda(t) \overline{h}^2_{j,\lambda,\tau}(x,t)(\varphi_\tau(t)+\zeta(t))(\partial_{x_j}H(x,y_\tau)-v_{j,\tau}(x)) dx dt \\
&+\int_0^T\rho_\lambda(t)\left\langle \partial_th^1_{j,\lambda,\tau}(\cdot,t),\overline{h}^2_{j,\lambda,\tau}(\cdot,t)\right\rangle_{H^{-1}(\Omega),H^1_0(\Omega)}dt .\end{align*}
Next we bound each term separately. By Proposition \ref{prop: w for rho case} and \ref{prop: w adj for rho case}, we have 
\begin{align*}
   &\quad \left\lvert \int_\Sigma (\Lambda_\lambda g_{j,\tau})\overline{g}_{j,\tau} d\sigma(x) dt\right \lvert \leq \|\Lambda_\lambda g_{j,\tau}\|_{H^{-1/2,-1/2}(S\times (0,T))} \|\overline{g}_{j,\tau}\|_{H^{1/2,1/2}_0(S\times (0,T))} \\
    & \leq C\|\Lambda_\lambda\|_{op} \|g_{j,\tau}\|_{H^{1/2,1/2}_0(S\times(0,T))}(\tau^{1-r-\frac{n}{2}}+\tau^{-\frac{n}{2}}) 
  \leq C\eta (\tau^{1-r-n}+\tau^{-n}) .
\end{align*}
Then by Theorem \ref{thm:est for gamma}, we find 
\begin{align*}
   &\quad \left \lvert \int_Q \gamma_\lambda(t) \nabla w^1_{j,\lambda,\tau} \cdot \nabla \overline{w}^2_{j,\lambda,\tau} dx dt\right \lvert  \leq C\eta  \|w^1_{j,\lambda,\tau}\|_{L^2(0,T;H^1(\Omega))} \|\overline{w}^2_{j,\lambda,\tau}\|_{L^2(0,T;H^1(\Omega))} \\
    & \leq C\eta \tau^{-\frac{n}{2}} (\tau^{-\frac{n}{2}} + \tau^{-r} + \tau^{2-r-\frac{n}{2}})
    \leq C \eta (\tau^{-n} + \tau^{2-r-n}) .
\end{align*}
Lastly, by Propositions \ref{prop: w for rho case} and \ref{prop: w adj for rho case} and repeating the argument for Theorem \ref{thm:est for gamma}, we obtain
$$|B| \leq C(\tau^{1-\frac{n}{2}}+\tau^{1-r-\frac{n}{2}}+\tau^{3-r-n}+\tau^{-r}+\tau^{4-r-n})  \leq C(\tau^{1-r-\frac{n}{2}}+\tau^{3-r-n}) . $$
Hence, combining the preceding estimates gives
$$ \begin{aligned}&\left\lvert \int_0^T \rho_\lambda(t) \varphi_\tau^2(t)dt \right \lvert \left\lvert \int_\Omega (\partial_{x_j} H(x,y_\tau))^2dx \right\lvert 
\leq C\left[\eta (\tau^{-n}+\tau^{1-r-n}) + \tau^{1-r-\frac{n}{2}}+\tau^{3-r-n} \right] ,\end{aligned}
$$
and summing the expression over $j=1,\ldots,n$ leads to
\begin{align*}&\left\lvert \int_0^T \rho_\lambda(t) \varphi_\tau^2(t)dt  \right \lvert \left\lvert \int_\Omega |\nabla H(x,y_\tau)|^2dx \right\lvert
\leq C\left[\eta (\tau^{-n}+\tau^{1-r-n}) +\tau^{1-r-\frac{n}{2}}+ \tau^{3-r-n} \right] .
\end{align*}
Using the estimate \eqref{eqn:lower bound for gradH} and dividing both sides of the last inequality by $\tau^{2-n}$, we get
\begin{equation}\label{esti4}
     \left\lvert \int_0^T \rho_\lambda(t) \varphi_\tau^2(t)dt \right \lvert \leq C(\eta(\tau^{-2}+\tau^{-1-r})+\tau^{-1-r+\frac{n}{2}} +\tau^{1-r}) .
\end{equation}
Since $\|\varphi_\tau\|_{L^2(\R)}=1$ and supp$(\varphi_\tau)\subset(0,T)$, applying \eqref{cond} and mean value theorem gives
\begin{align*}\left\lvert \int_0^T \rho_\lambda(t) \varphi_\tau^2(t) dt \right \lvert&\geq |\rho_\lambda(t_0)|-\left\lvert \int_0^T \rho_\lambda(t) \varphi_\tau^2(t)dt -\rho_\lambda(t_0)\right\lvert\\
&\geq |\rho_\lambda(t_0)|- \int_\R |\rho_\lambda(s\tau^r+t_0)-\rho_\lambda(t_0)| \varphi^2(s)ds \\
&\geq |\rho_\lambda(t_0)|- 2\bigg(\sup_{\mu\in[-R,R]}\kappa(\mu)\bigg) \tau^r \int_{-1}^1 s \varphi^2(s)ds\geq  |\rho_\lambda(t_0)|-C\tau^r.
\end{align*}
Combining this with \eqref{esti4} yields
$|\rho_\lambda(t_0)|\leq C(\tau^r + 
\eta(\tau^{-2}+\tau^{-1-r}) + \tau^{-1-r+\frac{n}{2}}+ \tau^{1-r} )$,
and by fixing $r = \frac{1}{4}$, we  obtain
$$
 |\rho_\lambda(t_0)|\leq C(\tau^{\frac{1}{4}}+\eta\tau^{-2}) .   
$$
Using this estimate, we shall prove \eqref{estt} for all possible values of $\eta\geq0$.\\
For $\eta=0$, by sending $\tau\to0$, we get $|\rho_\lambda(t_0)|=0$ which, combined with \eqref{cond3}, clearly implies \eqref{estt}. For $\eta\in \left(0,\tau_0^{4/9}\right)$, we can choose $\tau=\eta^{4/9}$ and deduce 
$$  |\rho_\lambda(t_0)|\leq C\eta^{\frac{1}{9}} , $$
which, combined with \eqref{cond3}, clearly implies \eqref{estt}. Moreover, for $\eta\geq\tau_0^{4/9}$, we have
$$   |\rho_\lambda(t_0)|\leq     |\rho_\lambda(t_0)|\frac{\eta^{\frac{1}{9}}}{\tau_0^{\frac{4}{81}}}\leq C\eta^{\frac{1}{9}}. $$
Thus \eqref{estt} holds for all $\eta\geq0$. This completes the proof of the theorem.

\end{proof}

\bibliographystyle{siam} 
\bibliography{References}

\begin{thebibliography}{10}

\bibitem{Al}
{\sc O.~M. Alifanov}, {\em Inverse heat transfer problems. {Transl}. from the {Russian}}, Berlin: Springer, 1994.

\bibitem{BBC}
{\sc J.~V. Beck, B.~Blackwell, and C.~R. S.~j. Clair}, {\em Inverse heat conduction. {Ill}-posed problems}.
\newblock A {Wiley}-{Interscience} {Publication}. {New} {York} etc.: {John} {Wiley} \& {Sons}, {Inc}. {XVII}, 308 p. ({TUB}/{Stat}: {B} 86233) (1985)., 1985.

\bibitem{Ca}
{\sc J.~R. Cannon}, {\em Determination of the unknown coefficient {{\(k(u)\)}} in the equation {{\(\nabla\cdot k(u)\nabla u = 0\)}} from overspecified boundary data}, J. Math. Anal. Appl., 18 (1967), pp.~112--114.

\bibitem{CD1}
{\sc J.~R. Cannon and P.~DuChateau}, {\em Determination of the conductivity of an isotropic medium}, J. Math. Anal. Appl., 48 (1974), pp.~699--707.

\bibitem{CD2}
\leavevmode\vrule height 2pt depth -1.6pt width 23pt, {\em An inverse problem for a nonlinear diffusion equation}, SIAM J. Appl. Math., 39 (1980), pp.~272--289.

\bibitem{CFKKU}
{\sc C.~I. C{\^a}rstea, A.~Feizmohammadi, Y.~Kian, K.~Krupchyk, and G.~Uhlmann}, {\em The {Calder{\'o}n} inverse problem for isotropic quasilinear conductivities}, Adv. Math., 391 (2021), p.~31.
\newblock Id/No 107956.

\bibitem{CLLO}
{\sc C.~I. C{\^a}rstea, M.~Lassas, T.~Liimatainen, and L.~Oksanen}, {\em An inverse problem for the {Riemannian} minimal surface equation}, J. Differ. Equations, 379 (2024), pp.~626--648.

\bibitem{Cho}
{\sc M.~Choulli}, {\em Stable determination of the nonlinear term in a quasilinear elliptic equation by boundary measurements}, C. R., Math., Acad. Sci. Paris, 361 (2023), pp.~1455--1470.

\bibitem{CK}
{\sc M.~Choulli and Y.~Kian}, {\em Logarithmic stability in determining the time-dependent zero order coefficient in a parabolic equation from a partial {Dirichlet}-to-{Neumann} map. application to the determination of a nonlinear term}, J. Math. Pures Appl. (9), 114 (2018), pp.~235--261.

\bibitem{COY}
{\sc M.~Choulli, E.~M. Quhabaz, and M.~Yamamoto}, {\em Stable determination of a semilinear term in a parabolic equation}, Commun. Pure Appl. Anal., 5 (2006), pp.~447--462.

\bibitem{egger2018tikhonov}
{\sc H.~Egger and B.~Hofmann}, {\em Tikhonov regularization in hilbert scales under conditional stability assumptions}, Inverse Problems, 34 (2018), p.~115015.

\bibitem{EPS3}
{\sc H.~Egger, J.-F. Pietschmann, and M.~Schlottbom}, {\em Identification of nonlinear heat conduction laws}, J. Inverse Ill-Posed Probl., 23 (2015), pp.~429--437.

\bibitem{EPS2}
\leavevmode\vrule height 2pt depth -1.6pt width 23pt, {\em On the uniqueness of nonlinear diffusion coefficients in the presence of lower order terms}, Inverse Probl., 33 (2017), p.~16.
\newblock Id/No 115005.

\bibitem{evans2022partial}
{\sc L.~C. Evans}, {\em Partial differential equations}, vol.~19, American Mathematical Society, 2022.

\bibitem{FA}
{\sc A.~Feizmohammadi}, {\em An inverse boundary value problem for isotropic nonautonomous heat flows}, Math. Ann., 388 (2024), pp.~1569--1607.

\bibitem{FKU}
{\sc A.~Feizmohammadi, Y.~Kian, and G.~Uhlmann}, {\em An inverse problem for a quasilinear convection-diffusion equation}, Nonlinear Anal., Theory Methods Appl., Ser. A, Theory Methods, 222 (2022), p.~30.
\newblock Id/No 112921.

\bibitem{gilbarg1977elliptic}
{\sc D.~Gilbarg and N.~S. Trudinger}, {\em Elliptic partial differential equations of second order}, vol.~224, Springer, 1977.

\bibitem{Is}
{\sc V.~Isakov}, {\em On uniqueness in inverse problems for semilinear parabolic equations}, Arch. Ration. Mech. Anal., 124 (1993), pp.~1--12.

\bibitem{Is1}
{\sc V.~Isakov}, {\em Uniqueness of recovery of some quasilinear partial differential equations}, Commun. Partial Differ. Equations, 26 (2001), pp.~1947--1973.

\bibitem{Is2}
\leavevmode\vrule height 2pt depth -1.6pt width 23pt, {\em Uniqueness of recovery of some systems of semilinear partial differential equations}, Inverse Probl., 17 (2001), pp.~607--618.

\bibitem{kachalov2001inverse}
{\sc A.~Kachalov, Y.~Kurylev, and M.~Lassas}, {\em Inverse boundary spectral problems}, Chapman and Hall/CRC, 2001.

\bibitem{kalf1992ee}
{\sc H.~Kalf}, {\em On ee levi’s method of constructing a fundamental solution for second-order elliptic equations}, Rendiconti del Circolo Matematico di Palermo, 41 (1992), pp.~251--294.

\bibitem{kian2023lipschitz}
{\sc Y.~Kian}, {\em Lipschitz and h{\"o}lder stable determination of nonlinear terms for elliptic equations}, Nonlinearity, 36 (2023), p.~1302.

\bibitem{Ki24}
\leavevmode\vrule height 2pt depth -1.6pt width 23pt, {\em Determination of quasilinear terms from restricted data and point measurements}, J. Funct. Anal., 287 (2024), p.~26.
\newblock Id/No 110612.

\bibitem{KLL}
{\sc Y.~Kian, T.~Liimatainen, and Y.-H. Lin}, {\em On determining and breaking the gauge class in inverse problems for reaction-diffusion equations}, Forum Math. Sigma, 12 (2024), p.~42.
\newblock Id/No e25.

\bibitem{KU}
{\sc Y.~Kian and G.~Uhlmann}, {\em Recovery of nonlinear terms for reaction diffusion equations from boundary measurements}, Arch. Ration. Mech. Anal., 247 (2023), p.~20.
\newblock Id/No 6.

\bibitem{ladyzhenskaia1968linear}
{\sc O.~A. Ladyzhenskaia, V.~A. Solonnikov, and N.~N. Ural'tseva}, {\em Linear and quasi-linear equations of parabolic type}, vol.~23, American Mathematical Soc., 1968.

\bibitem{LM2}
{\sc J.~Lions and E.~Magenes}, {\em Non-Homogeneous Boundary Value Problems and Applications: Vol. 2}, Springer, 1972.

\bibitem{LM1}
{\sc J.~L. Lions and E.~Magenes}, {\em Non-homogeneous boundary value problems and applications: Vol. 1}, vol.~181, Springer Science \& Business Media, 2012.

\bibitem{MU}
{\sc C.~Mu{\~n}oz and G.~Uhlmann}, {\em The {Calder{\'o}n} problem for quasilinear elliptic equations}, Ann. Inst. Henri Poincar{\'e}, Anal. Non Lin{\'e}aire, 37 (2020), pp.~1143--1166.

\bibitem{Nu1}
{\sc J.~Nurminen}, {\em An inverse problem for the minimal surface equation}, Nonlinear Anal., Theory Methods Appl., Ser. A, Theory Methods, 227 (2023), p.~19.
\newblock Id/No 113163.

\bibitem{Nu2}
\leavevmode\vrule height 2pt depth -1.6pt width 23pt, {\em An inverse problem for the minimal surface equation in the presence of a {Riemannian} metric}, Nonlinearity, 37 (2024), p.~22.
\newblock Id/No 095029.

\bibitem{PR}
{\sc M.~Pilant and W.~Rundell}, {\em A uniqueness theorem for determining conductivity from overspecified boundary data}, J. Math. Anal. Appl., 136 (1988), pp.~20--28.

\bibitem{Schuster}
{\sc D.~Rothermel, T.~Schuster, R.~Schorr, and M.~Peglow}, {\em Determination of the temperature-dependent thermal material properties in the cooling process of steel plates}, Math. Probl. Eng., 2021 (2021), p.~13.
\newblock Id/No 6653388.

\bibitem{SKFS}
{\sc W.~Sch\"utz, H.~Kirsch, P.~Fl\"uss, and A.~Streisselberger}, {\em Extended property combinations in thermomechanically control processed steel plates by application of advanced rolling and cooling technology}, Ironmaking and Steelmaking, 28 (2001), pp.~180--184.

\bibitem{serizawa2015plate}
{\sc Y.~Serizawa, S.~Nakagawa, Y.~Kadoya, R.~Yamamoto, H.~Ueno, Y.~Haraguchi, H.~Tachibana, T.~Iwaki, and T.~Oda}, {\em Plate cooling technology for the thermo mechanical control process (tmcp) in nippon steel \& sumitomo metal corporation}, Nippon Steel \& Sumitomo Metal Technical Report, 110 (2015), pp.~17--24.

\bibitem{Su}
{\sc Z.~Sun}, {\em On a quasilinear inverse boundary value problem}, Math. Z., 221 (1996), pp.~293--305.

\bibitem{SuU}
{\sc Z.~Sun and G.~Uhlmann}, {\em Inverse problems in quasilinear anisotropic media}, Am. J. Math., 119 (1997), pp.~771--797.

\bibitem{werner2019convergence}
{\sc F.~Werner and B.~Hofmann}, {\em Convergence analysis of (statistical) inverse problems under conditional stability estimates}, Inverse Problems, 36 (2019), p.~015004.

\end{thebibliography}

\end{document}